\documentclass[10pt]{amsart}
\usepackage{amsmath}
\usepackage{amsthm}
\usepackage{amsfonts}
\usepackage{mathrsfs}
\usepackage{amssymb}
\usepackage{amscd}
\usepackage{pgf}
\usepackage{tikz}
\usetikzlibrary{cd}
\usetikzlibrary{decorations.pathreplacing, decorations.markings, calc}

\usepackage{amscd}

    \newtheorem{Lem}{Lemma}[section]
    \newtheorem{Lem-Def}{Lemma-Definition}[section]
    \newtheorem{Prop}[Lem]{Proposition}
    \newtheorem{Alg}[Lem]{Algorithm}
    \newtheorem{Conjecture}[Lem]{Conjecture}

    \newtheorem{Thm}[Lem]{Theorem}
    \newtheorem{Cor}[Lem]{Corollary}
		\newtheorem*{Thm*}{Theorem}

\theoremstyle{definition}

    \newtheorem{Def}[Lem]{Definition}
    \newtheorem{Exa}[Lem]{Example}
    \newtheorem{Rem}[Lem]{Remark}

\expandafter\let\csname ver@amsthm.sty\endcsname\relax
\let\theoremstyle\relax

\usepackage[utf8]{inputenc}
\usepackage[
    breaklinks,
    colorlinks,
    citecolor=teal,
    linkcolor=teal,
    urlcolor=teal,
    pagebackref=true,
    hyperindex
    ]{hyperref}
\usepackage{fancyhdr}
\usepackage[
    hscale=0.7,
    vscale=0.75,
    headheight=13pt,
    centering,
    ]{geometry}
\usepackage{amsmath}
\usepackage{amsthm}
\usepackage{amssymb}
\usepackage{mathtools}
\usepackage{stmaryrd}
\usepackage{mathdots}
\usepackage[all,cmtip]{xy}
\usepackage{xcolor}
\usepackage{framed}
\usepackage{pgffor}
\usepackage{fnpct}
\usepackage[capitalize]{cleveref}

\theoremstyle{plain}

\theoremstyle{definition}

\numberwithin{figure}{section}

\numberwithin{equation}{section}

\makeatletter

\def\@myMR[#1 #2]{\relax\ifhmode\unskip\spacefactor3000 \space\fi
  \MRhref{#1}{MR\,#1}}
\renewcommand\MR[1]{\@myMR[#1 ]}
\renewcommand{\MRhref}[2]{{\tiny%
  \href{http://www.ams.org/mathscinet-getitem?mr=#1}{#2}}%
}
\renewcommand*{\backref}[1]{}
\renewcommand*{\backrefalt}[4]{%
    \tiny%
    ({
    \ifcase #1 not cited%
          \or cit.\ on p.~#2%
          \else cit.\ on pp.~#2%
    \fi%
    })\\[-.6em]}

\def\maketitle{\par
  \@topnum\z@ 
  \@setcopyright
  \thispagestyle{empty}
  \ifx\@empty\shortauthors \let\shortauthors\shorttitle
  \else \andify\shortauthors
  \fi
  \@maketitle@hook
  \begingroup
  \@maketitle
  \toks@\@xp{\shortauthors}\@temptokena\@xp{\shorttitle}%
  \toks4{\def\\{ \ignorespaces}}
  \edef\@tempa{%
    \@nx\markboth{\the\toks4
      \@nx\MakeUppercase{\the\toks@}}{\the\@temptokena}}%
  \@tempa
  \endgroup
  \c@footnote\z@
    \renewcommand{\footnoterule}{%
      \kern -3pt
      \hrule width \textwidth height .5pt
      \kern 2pt
    }
  {
    \renewcommand\thefootnote{}
    \vspace{-2em}
    \footnote{
      \par\vspace{-1.2em}\noindent
      \def\@footnotetext##1{\noindent{\footnotesize##1}\par}%
      \let\@makefnmark\relax  \let\@thefnmark\relax
      \ifx\@empty\@date\else \@footnotetext{\@setdate}\fi
      \ifx\@empty\@subjclass\else \@footnotetext{\@setsubjclass}\fi
      \ifx\@empty\@keywords\else \@footnotetext{\@setkeywords}\fi
      \ifx\@empty\thankses\else \@footnotetext{%
        \def\par{\let\par\@par}\@setthanks}%
      \fi
    }
    \addtocounter{footnote}{-1}
  }
  \@cleartopmattertags
}

\def\@adminfootnotes{\@empty}

\def\@settitle{\begin{center}%
  \baselineskip14\p@\relax
    \bfseries
\Large
  \@title
  \end{center}%
}

\def\@setauthors{%
  \begingroup
  \def\thanks{\protect\thanks@warning}%
  \trivlist
  \centering\footnotesize \@topsep30\p@\relax
  \advance\@topsep by -\baselineskip
  \item\relax
  \author@andify\authors
  \def\\{\protect\linebreak}%
  \large{\authors}%
  \ifx\@empty\contribs
  \else
    ,\penalty-3 \space \@setcontribs
    \@closetoccontribs
  \fi
  \endtrivlist
  \endgroup
}

\def\@setaddresses{\par
  \nobreak \begingroup
\footnotesize
  \def\author##1{\end{minipage}\hskip 1sp \begin{minipage}{.5\textwidth}\raggedright%
    ~\\[2em]{\bf##1}\\[.5em]%
  }%
  \interlinepenalty\@M
  \def\address##1##2{\begingroup
    {\ignorespaces##2}\endgroup\\[.5em]}%
  \def\curraddr##1##2{\begingroup
    \@ifnotempty{##2}{\nobreak\indent\curraddrname
      \@ifnotempty{##1}{, \ignorespaces##1\unskip}\/:\space
      ##2\par}\endgroup}%
  \def\email##1##2{\begingroup
    \@ifnotempty{##2}{\nobreak\indent
      \@ifnotempty{##1}{, \ignorespaces##1\unskip}
      \ttfamily##2\par}\endgroup}%
  \def\urladdr##1##2{\begingroup
    \def~{\char`\~}%
    \@ifnotempty{##2}{\nobreak\indent\urladdrname
      \@ifnotempty{##1}{, \ignorespaces##1\unskip}\/:\space
      \ttfamily##2\par}\endgroup}%
  \setlength{\parindent}{0pt}%
  \vfill%
  {
  \begin{minipage}{0mm}
  \addresses
  \end{minipage}
  }
  \endgroup
}

\renewcommand{\author}[2][]{%
  \ifx\@empty\authors
    \gdef\authors{#2}%
    \g@addto@macro\addresses{\author{#2}}%
  \else
    \g@addto@macro\authors{\and#2}%
    \g@addto@macro\addresses{\author{#2}}%
  \fi
  \@ifnotempty{#1}{%
    \ifx\@empty\shortauthors
      \gdef\shortauthors{#1}%
    \else
      \g@addto@macro\shortauthors{\and#1}%
    \fi
  }%
}
\edef\author{\@nx\@dblarg
  \@xp\@nx\csname\string\author\endcsname}

\def\@secnumfont{\@empty}

\def\section{\@startsection{section}{1}%
  \z@{.7\linespacing\@plus\linespacing}{.5\linespacing}%
  {\large\bfseries\centering}}

\makeatother

\usepackage{enumerate}
\usepackage{hyperref}

\newcommand{\A}{\mathcal A}

\newcommand{\col}{\colon}

\newcommand{\ol}{\overline}

\newcommand{\D}{\mathcal{D}}

\DeclareMathOperator{\wt}{wt}
\DeclareMathOperator{\csf}{csf}
\DeclareMathOperator{\asc}{asc}

\newcommand{\sym}{\Lambda}

\title[Chromatic symmetric functions from the modular law]{Chromatic symmetric functions from the modular law}

\author{Alex Abreu}

\address{
    Instituto de Matemática e Estatística\\
    Universidade Federal Fluminense\\
    Rua Prof. M. W. de Freitas, S/N\\
    24210-201 Niterói, Rio de Janeiro, Brasil
}

\email{alexbra1@gmail.com}

\author{Antonio Nigro}

\address{
    Instituto de Matemática e Estatística\\
    Universidade Federal Fluminense\\
    Rua Prof. M. W. de Freitas, S/N\\
    24210-201 Niterói, Rio de Janeiro, Brasil
}

\email{antonio.nigro@gmail.com}

\newcommand{\dyckpath}[2]{
\draw[line width=2pt] (#1) foreach \dir in {#2}{ -- ++(\dir*90:1)};
}
\newcommand{\dyckpathc}[3]{
\draw[line width=2pt, color=#3] (#1) foreach \dir in {#2}{ -- ++(\dir*90:1)};
}
\newcommand{\gline}[4]{
\draw[line width=#4 pt, color=#3] (#1) -- (#2);
}
\newcommand{\glinear}[4]{
\draw[line width=#3 pt, color=#2,postaction=decorate] (#1) -- +(0,-1);
\node at ($(#1)+(-0,-0.5)$) { #4};
}
\newcommand{\glinearr}[4]{
\draw[line width=#3 pt, color=#2,postaction=decorate] (#1) -- +(-1,-1);
\node at ($(#1)+(-0.5,-0.5)$) { #4};
}

\begin{document}
\maketitle

\begin{abstract}
  In this article we show how to compute the chromatic quasisymmetric function of indifference graphs from the modular law introduced in \cite{GPmodular}. We provide an algorithm which works for any function that satisfies this law, such as unicellular LLT polynomials. When the indifference graph has bipartite complement it reduces to a planar network, in this case, we prove that the coefficients of the chromatic quasisymmetric function in the elementary basis are positive unimodal polynomials and characterize them as certain $q$-hit numbers (up to a factor). Finally, we discuss the logarithmic concavity of the coefficients of the chromatic quasisymmetric function.
\end{abstract}

\tableofcontents

\section{Introduction}
    The chromatic polynomial can be characterized as the unique function
    \[
    \chi\col \mathbf{Graphs}\to \mathbb{Q}[x]
    \]
    that has the following three properties.\footnote{Actually, only properties \eqref{item:A} and \eqref{item:C} are needed.}
    \begin{enumerate}[(A)]
        \item\label{item:A} It satisfies the deletion-contraction recurrence, $\chi_{G}=\chi_{G\setminus e}-\chi_{G/e}$ for every edge $e\in E(G)$.
        \item\label{item:B} It is multiplicative, $\chi_{G_1\sqcup G_2}=\chi_{G_1}\chi_{G_2}$.
        \item\label{item:C} It has values at complete graphs given by $\chi_{K_n}(x)=x(x-1)\cdots(x-n+1)$.
    \end{enumerate}

     The chromatic polynomial of a graph admits a symmetric function generalization introduced by Stanley in \cite{Stan95}. Given a graph $G$ it is defined as
     \[
     \csf(G):=\sum_{\kappa}x_{\kappa}
     \]
     where the sum runs through all proper colorings of the vertices  $\kappa\col V(G)\to \mathbb{N}$ and $x_{\kappa}:=\prod_{v\in V(G)} x_{\kappa(v)}$. A coloring $\kappa$ is proper if $\kappa(v)\neq\kappa(v')$ whenever $v$ and $v'$ are adjacent. If $\sym$ is the algebra of symmetric functions, it turns out that $\csf$ is a function from $\mathbf{Graphs}$ to $\Lambda$. This function is multiplicative and its values at complete graphs are given by $\csf(K_n)=n!e_n$ (where $e_n$ is the elementary symmetric function of degree $n$).  However, it does not satisfy the deletion-contraction recurrence, one simple reason being that the chromatic symmetric function is homogeneous of degree equal to the number of vertices of $G$.\par

       In this paper we will restrict ourselves to indifference graphs, i.e., graphs whose set of vertices can be identified with $[n]:=\{1,2,\ldots, n\}$ and such that  if $\{i,j\}$ is an edge with $i<j$, then $\{i,k\}$ and $\{k,j\}$ are also edges for every $k$ such that $i<k<j$. This class of graphs, restrictive as it might look, ends up having deep relations with geometry and representation theory, see, for example, \cite{BrosnanChow}, \cite{GP} and \cite{abe2019}. \par

       Indifference graphs can be naturally associated with Hessenberg functions and Dyck paths (see Figure \ref{fig:fig1}). A Hessenberg function is a non-decreasing function $h\col [n]\to[n]$ such that $h(i)\geq i$ for every $i\in [n]$. The graph associated to $h$ is the graph with vertex set $[n]$ and set of edges $E=\{\{i,j\};i<j\leq h(i)\}$. All indifference graphs arise from Hessenberg functions. To each Hessenberg function there is an associated Dyck path, which is the unique path with $h(i)$ north steps before the $i$-th east step.  We usually denote a Hessenberg function by the $n$-tuple of its values $h=(h(1),h(2),\ldots, h(n))$ or by the word in $n$ (north step) and $e$ (east step) corresponding to its associated Dyck path.  We denote by $\D$ the set of Dyck paths, which will be identified with the set of Hessenberg functions and with the set of indifference graphs. In the rest of the introduction by \emph{graph} we will always mean an indifference graph.
       \begin{figure}[htb]  \centering
       \begin{tikzpicture}
       \begin{scope}[scale=0.6]
       \draw[help lines] (0,0) grid +(3,3);
       \draw[fill=blue, fill opacity=0.2, draw opacity=0] (0,1) rectangle +(1,1);
       \draw[fill=red, fill opacity=0.2, draw opacity=0] (1,2) rectangle +(1,1);
       \dyckpath{0,0}{1,1,0,1,0,0}
       \node at (1.5,-1) {$h=(2,3,3)$};
       \node at (1.75,-1.7) {$=nnenee$};
       \end{scope}
       \begin{scope}[scale=0.6, shift={(4,0)}]
        \node (1) at (1,1.5) [label=below:$1$, shape=circle, fill=black, inner sep=2pt] {};
        \node (2) at (3,1.5) [label=below:$2$, shape=circle, fill=black,inner sep=2pt] {};
        \node (3) at (5,1.5) [label=below:$3$, shape=circle, fill=black,inner sep=2pt] {};
        \node at (3,-1) {$G$};
        \draw[color=blue] (1) to (2);
        \draw[color=red] (2) to (3);
        \end{scope}
        \end{tikzpicture}
        \caption{A Hessenberg function $h$, its corresponding Dyck path and its associated indifference Graph.}
       \label{fig:fig1}
   \end{figure}
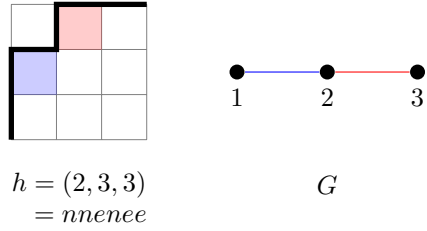

     When $G$ is the graph associated to $h$, we can recover the chromatic polynomial of $G$ by a differential operator $\Delta_h$ in the Weyl algebra $\mathbb{Q}[x,\partial]$. The operator $\Delta_h$ is obtained from $h$ by replacing each east step with $\partial$ and each north step with $x$. For example, if $h=nnenee$, then $\Delta_h=x^2\partial x\partial^2$. It is not hard to check that the chromatic polynomial of $G$ satisfies the following equality
   \[
   \Delta_{h}x^n=\chi_G(n)x^n.
   \]

   Moreover, with this interpretation, the deletion-contraction recurrence (when applied to edges of $G$ that correspond to corners in $h$) is, essentially, the well known formula $[\partial,x]=1$, which gives
   \begin{equation}
       \label{eq:partialx}
       x\partial=\partial x- 1.
   \end{equation}
   See Figure \ref{fig:delcontpartial}.
    \begin{figure}[htb]  \centering
    \begin{tikzpicture}
    \begin{scope}[scale=0.6]
       \draw[help lines] (0,0) grid +(3,3);
       \dyckpath{0,0}{1,1,0,1,0,0}
       \dyckpathc{0,1}{1,0}{yellow}
       \node at (1.5,-1) {$h=nnenee$};
       \node at (-0.3,1.5) {$x$};
       \node at (0.5,2.3) {$\partial$};
       \end{scope}
       \begin{scope}[scale=0.6, shift={(4,0)}]
        \node (1) at (1,1.5) [label=below:$1$, shape=circle, fill=black, inner sep=2pt] {};
        \node (2) at (3,1.5) [label=below:$2$, shape=circle, fill=black,inner sep=2pt] {};
        \node (3) at (5,1.5) [label=below:$3$, shape=circle, fill=black,inner sep=2pt] {};
        \node at (3,-1) {$G$};
        \draw (1) to (2);
        \draw (2) to (3);
        \end{scope}
        \begin{scope}[scale=0.6,shift={(-6,-5)}]
       \begin{scope}
       \draw[help lines] (0,0) grid +(3,3);
       \dyckpath{0,0}{1,0,1,1,0,0}
       \dyckpathc{0,1}{0,1}{yellow}
       \node at (1.5,-1) {$h'=nennee$};
       \node at (0.5,0.7) {${\partial}$};
       \node at (1.3,1.5) {$x$};
       \end{scope}
       \begin{scope}[ shift={(4,0)}]
        \node (1) at (1,1.5) [label=below:$1$, shape=circle, fill=black, inner sep=2pt] {};
        \node (2) at (3,1.5) [label=below:$2$, shape=circle, fill=black,inner sep=2pt] {};
        \node (3) at (5,1.5) [label=below:$3$, shape=circle, fill=black,inner sep=2pt] {};
        \node at (3,-1) {$G\setminus e$};
        \draw (2) to (3);
        \end{scope}
        \begin{scope}[shift={(12,0)}]
        \begin{scope}
       \draw[help lines] (0,0) grid +(2,2);
       \dyckpath{0,0}{1,1,0,0}
       \node at (1.5,-1) {$h''=nnee$};
       \end{scope}
       \begin{scope}[shift={(4,0)}]
        \node (1) at (1,1.5) [label=below:${1,2}$, shape=circle, fill=black, inner sep=2pt] {};
        \node (2) at (3,1.5) [label=below:$3$, shape=circle, fill=black,inner sep=2pt] {};
        \node at (2,-1) {$G/e$};
        \draw (1) to (2);
        \end{scope}
        \end{scope}
        \end{scope}
    \end{tikzpicture}
   \caption{The deletion-contraction recurrence}
   \label{fig:delcontpartial}
   \end{figure}
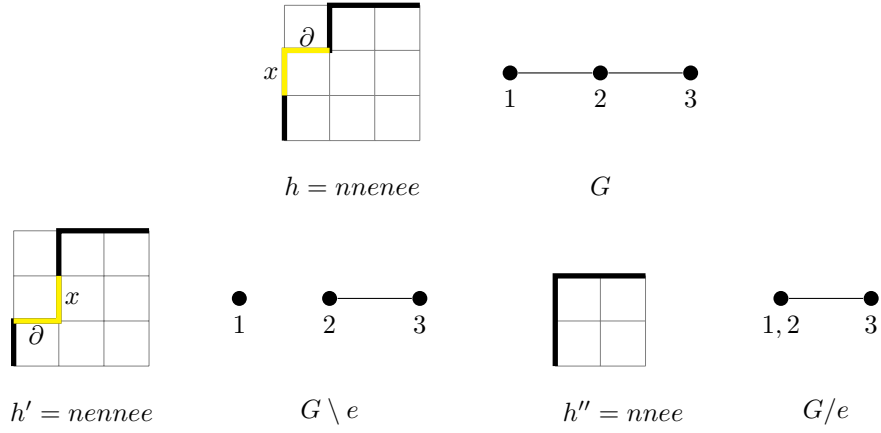

  As pointed out before, the fact that Formula \eqref{eq:partialx} is not homogeneous is one of the reasons for the deletion-contraction recurrence not holding for the chromatic symmetric function. On the other hand, it is not hard to find homogeneous relations for $\partial$ and $x$. For example, one can simply consider $[[\partial,x],x]=0$ and $[[\partial,x],\partial]=0$. Explicitly this gives
   \begin{equation}
   \label{eq:basicpartial}
       \begin{aligned}
       2x\partial x&=x^2\partial+\partial x^2\\
       2\partial x\partial&=\partial^2x+x\partial^2.
       \end{aligned}
    \end{equation}
    In particular, if $G_0$, $G_1$ and $G_2$ are graphs associated to Dyck paths $h_0$, $h_1$ and $h_2$ such that $h_0$ and $h_2$ are obtained from $h_1$ by replacing a subpath $nen$ with $enn$ and $nne$,  respectively, then (see Figure \ref{fig:xpartialx} below)\par
    \begin{equation}
    \label{eq:basicpol}
        2\chi_{G_1}=\chi_{G_0}+\chi_{G_2}.
    \end{equation}
    Similarly, the same holds if $h_0$ and $h_2$ are obtained from $h_1$ by replacing a subpath $ene$ with $een$ and $nee$, respectively.


     \begin{figure}[htb]  \centering
    \begin{tikzpicture}
    \begin{scope}[scale=0.6]
       \draw[help lines] (0,0) grid +(3,3);
       \dyckpath{0,0}{1,1,0,1,0,0}
       \dyckpathc{0,1}{1,0,1}{yellow}
       \node at (1.5,-1) {$h=nnenee$};
       \node at (-0.3,1.5) {$x$};
       \node at (0.5,2.3) {$\partial$};
       \node at (1.3,2.5) {$x$};
       \end{scope}
       \begin{scope}[scale=0.6, shift={(4,0)}]
        \node (1) at (1,1.5) [label=below:$1$, shape=circle, fill=black, inner sep=2pt] {};
        \node (2) at (3,1.5) [label=below:$2$, shape=circle, fill=black,inner sep=2pt] {};
        \node (3) at (5,1.5) [label=below:$3$, shape=circle, fill=black,inner sep=2pt] {};
        \node at (3,0) {$G_1$};
        \draw (1) to (2);
        \draw (2) to (3);
        \end{scope}
        \begin{scope}[scale=0.6,shift={(-6,-5)}]
       \begin{scope}
       \draw[help lines] (0,0) grid +(3,3);
       \dyckpath{0,0}{1,0,1,1,0,0}
       \dyckpathc{0,1}{0,1,1}{yellow}
       \node at (1.5,-1) {$h_0=nennee$};
       \node at (0.5,0.7) {${\partial}$};
       \node at (1.3,1.5) {$x$};
       \node at (1.3,2.5) {$x$};
       \end{scope}
       \begin{scope}[ shift={(4,0)}]
        \node (1) at (1,1.5) [label=below:$1$, shape=circle, fill=black, inner sep=2pt] {};
        \node (2) at (3,1.5) [label=below:$2$, shape=circle, fill=black,inner sep=2pt] {};
        \node (3) at (5,1.5) [label=below:$3$, shape=circle, fill=black,inner sep=2pt] {};
        \node at (3,0) {$G_0$};
        \draw (2) to (3);
        \end{scope}
        \begin{scope}[shift={(12,0)}]
        \begin{scope}
       \draw[help lines] (0,0) grid +(3,3);
       \dyckpath{0,0}{1,1,1,0,0,0}
       \dyckpathc{0,1}{1,1,0}{yellow}
       \node at (1.5,-1) {$h_2=nnneee$};
       \node at (0.5,3.3) {${\partial}$};
       \node at (-0.3,1.5) {$x$};
       \node at (-0.3,2.5) {$x$};
       \end{scope}
       \begin{scope}[shift={(4,0)}]
        \node (1) at (1,1.5) [label=below:$1$, shape=circle, fill=black, inner sep=2pt] {};
        \node (2) at (3,1.5) [label=below:$2$, shape=circle, fill=black,inner sep=2pt] {};
        \node (3) at (5,1.5) [label=below:$3$, shape=circle, fill=black,inner sep=2pt] {};
        \node at (3,0) {$G_2$};
        \draw (1) to (2);
        \draw (2) to (3);
        \draw (1) [bend left=45] to (3);
        \end{scope}
        \end{scope}
        \end{scope}
    \end{tikzpicture}
   \caption{Homogeneous relation involving $x$ and $\partial$.}
   \label{fig:xpartialx}
   \end{figure}
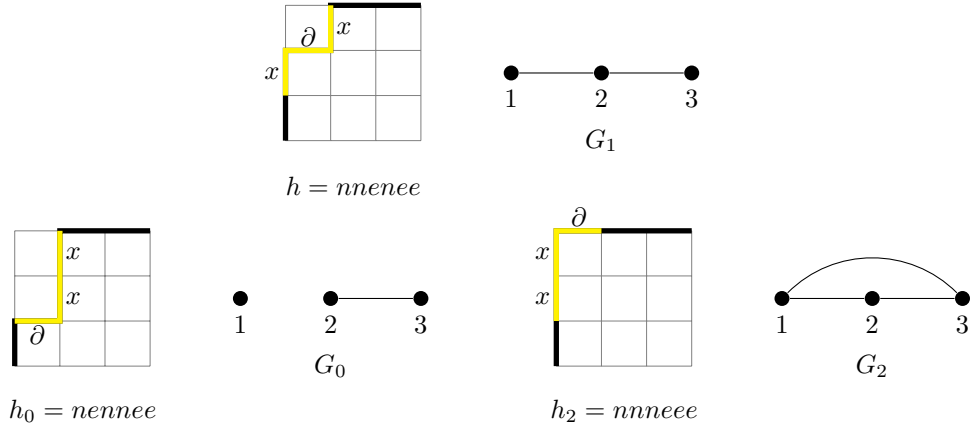
    One can actually replace Property \eqref{item:A} (the deletion-contraction recurrence) with the recurrence in Equation \eqref{eq:basicpol}, and these properties will still characterize the chromatic polynomial for indifference graphs. In other words, the restriction of $\chi$ to the set of indifference graphs is the unique function that has the following three properties.
    \begin{enumerate}[(A)]
        \item[(A')]\label{item:A'} Whenever $G_0$, $G_1$, and $G_2$ are graphs associated to Dyck paths $h_0$, $h_1$, and $h_2$ such that $h_0$ and $h_2$ are obtained from $h_1$ by replacing a subpath $nen$ with $enn$ and $nne$,  respectively, or by replacing a subpath $ene$ with $een$ and $nee$, respectively, then $2\chi_{G_1}=\chi_{G_0}+\chi_{G_2}$.
        \item[(B)]\label{item:B'} It is multiplicative, $\chi_{G_1\sqcup G_2}=\chi_{G_1}\chi_{G_2}$.
        \item[(C)]\label{item:C'} It has values at complete graphs given by $\chi_{K_n}(x)=x(x-1)\cdots(x-n+1)$.
    \end{enumerate}
    One possible way to see this is to find more general relations between $\partial$ and $x$ starting from Equation \eqref{eq:basicpartial}. One example is the following
   \begin{equation}
       \label{eq:relpartial}
       (b+1)x\partial^lx^b=l\partial^{l-1}x^{b+1}\partial+(b+1-l)\partial^lx^{b+1}.
   \end{equation}
   This equation, translated to graphs, means that when $G_0$, $G_1$ and $G_2$ are the graphs associated to Dyck paths $h_0$, $h_1$, and $h_2$ (see Figure \ref{fig:h0h2}) such that $h_0$ and $h_2$ are obtained from $h_1$ by replacing a subpath  $ne^ln^b$ with $e^ln^{b+1}$ and $e^{l-1}n^{b+1}e$, respectively, then we have
   \begin{equation}
       \label{eq:relpol}
       (b+1)\cdot \chi_{G_1}=l\cdot\chi_{G_2}+(b+1-l)\cdot\chi_{G_0}.
   \end{equation}

   Given a connected graph $G_1$, which is not complete,  we have that its associated Dyck path $h_1$ must end with $ne^ln^be^c$ for some positive integers $l,b,c$. If we define $h_0$ and $h_2$ by replacing the subpath $ne^ln^b$ of $h_1$ with $e^ln^{b+1}$ and $e^{l-1}n^{b+1}e$, respectively, we have that both $h_0$ and $h_2$ are Dyck paths (because $G_1$ is connected) and we can apply Equation \eqref{eq:relpol}. This process eventually ends, since $b$ will increase at each step.  Of course, the chromatic polynomial of indifference graphs can be easily computed directly, but the above recurrence is useful, for instance, if one wants to write the chromatic polynomial in certain bases as in \cite{Brenti}.\par

    The ideia is to repeat this process for the chromatic symmetric function.  Actually, we will work with the chromatic quasisymmetric function introduced by Shareshian and Wachs in \cite{ShareshianWachs}. For a graph $G$ with set of vertices $[n]$, the chromatic quasisymmetric function $\csf_q(G)$ is defined as
    \[
    \csf_q(G):=\sum_{\kappa}q^{\asc_G(k)} x_\kappa.
    \]
    where the sum runs through all proper colorings of $G$  and
    \[
    \asc_G(\kappa):=|\{(i,j);i<j, \kappa(i)<\kappa(j); \{i,j\}\in E(G)\}|
    \]
    is the number of ascents of the coloring $\kappa$.  When $G$ is an indifference graph, we have that $\csf_q(G)$ is actually symmetric, so we will think of $\csf_q$ as function $\csf_q\col \D\to \Lambda_q$, where $\Lambda_q$ is the algebra of symmetric functions with coefficients in $\mathbb{Q}(q)$. \par

     With a little experimentation one can see that Equation \eqref{eq:basicpol} does not hold in general for the chromatic symmetric function of indifference graphs. However, it still holds if we add some extra assumptions on $h_0$, $h_1$ and $h_2$, which are summarized in Definition \ref{def:modular}. The purpose of this article is to determine when Equation \eqref{eq:relpol} lifts to the chromatic symmetric function (see Proposition \ref{prop:relations}). Moreover, we prove that there are enough of these liftings  to fully characterize Stanley's chromatic symmetric function on indifference graphs as stated in the following theorem.


    \begin{Thm}
    The function $\csf_q\col \D\to \Lambda_q$ is the unique function that has the following three properties.
    \begin{enumerate}[(A)]
    \item\label{item:At} It satisfies the modular law, as in Definition \ref{def:modular}.
    \item\label{item:Bt} It is multiplicative, $\csf_q(G_1\sqcup G_2)=\csf_q(G_1)\csf_q(G_2)$.
    \item It has values at complete graphs given by $\csf_q(K_n)=n!_qe_n$.
    \end{enumerate}
    \end{Thm}

    Actually we prove a more general result. For an indifference graph $G$ with vertex set $[n]$, we denote by $G^t$ its transposed graph, that is, we relabel the vertices of $G$ via $i\mapsto n+1-i$.

    \begin{Thm}
    \label{thm:main}
    Let $A$ be a $\mathbb{Q}(q)$-algebra and let $f\col \D\to A$  be a function  that satisfies the modular law, as in Definition \ref{def:modular}. Then $f$ is determined by its values $f(K_{n_1}\sqcup K_{n_2}\sqcup\cdots\sqcup K_{n_m})$ at the disjoint ordered union of complete graphs and these values are independent of the order in which the union is taken. Moreover, we have that $f(G^{t})=f(G)$ for every indifference graph $G$.
    \end{Thm}
    Our proof of Theorem \ref{thm:main} is constructive. We find an algorithm (Algorithm \ref{alg:alg}) based exclusively on the modular law. This algorithm was implemented in SAGE and is available upon request. \par
    
    Every function that satisfies the modular law is intimately related with the chromatic symmetric function, as seen in Corollary \ref{cor:fcsf}. It would be interesting to find functions with combinatorial interpretations that satisfy the modular law. In \cite{ANtree} the authors define one such function enumerating increasing spanning forests and use it to sharpen the description of the $e$-coefficients of unicellular LLT polynomials conjectured in \cite{Alexandersson_2020} and \cite{garsia2019epositivity}. 

     When $G$ is an indifference graph whose complement is bipartite, we observe in Remark \ref{rem:network} that the algorithm reduces to a planar network. In this situation, we show that $\csf_q(G)$ can be computed in terms of $q$-hit numbers. This is a $q$-analogue of Stanley-Stembridge combinatorial formula \cite[Theorem 4.3]{StanStem}.  We recall that the partition associated to $h$ is given by $\lambda=(n-h(1),n-h(2),\ldots, n-h(n))$.

    \begin{Thm}
    Let $h$ be an Hessenberg function whose associated indifference graph has bipartite complement and let $\lambda$ be the partition associated to $h$. If $m:=\min\{\lambda_1,\ell(\lambda)\}$, then
    \[
    \csf_q(h)=m!_qR_{m,n-m}(\lambda)e_{n-m,m}+\sum_{j<m}q^jj!_q[m-2j]_qR_{j,n-j-1}(\lambda)e_{n-j,j}.
    \]
    where $R_{j,k}(\lambda)$ are the Garsia-Remmel $q$-deformation of hit numbers, that is, it enumerates weighted rook placements on $k\times k$ board with exactly $j$ rooks on the Young diagram of $\lambda$. Moreover, we have that $\csf_q(h)$ is $e$-unimodal.
    \end{Thm}
    
      In the last section we discuss the logarithmic concavity of the coefficients of the chromatic quasisymmetric function. In the breakthrough work \cite{Huh} it is proved that the chromatic polynomial of a graph is log-concave. This result was later generalized to matroids in \cite{AHK}. We supply some evidence supporting the logarithmic concavity of the $e$-coefficients of $\csf_q(h)$ for $h\in \D$.

     We point out that several analogues of deletion-contraction exist for the chromatic symmetric function (or some closely related symmetric functions). A  non-commutative chromatic symmetric function is defined in \cite{Sagan} which satisfies a deletion-contraction recurrence. In \cite{GPmodular}, a \emph{modular law} for the chromatic symmetric function is introduced for any graph. When restricted to indifference graphs it is the analogue of Equation \eqref{eq:basicpol}. In \cite{Lee}, this relation is found for the closely related unicellular LLT polynomial.  A chromatic symmetric function for weighted graphs is defined in \cite{crew} which satisfies a deletion-contraction recurrence when one considers contractions of weighted graphs.  Other linear relations in various settings can be found in \cite{OrellanaScott}, \cite{HNY},  \cite{dadderio}, and \cite{AS2020}. \par

      It is also worth mentioning that, for unicellular LLT polynomials, the analogy with differential operators was made precise in \cite{CarlssonMellit}. They defined operators $d_{-}$ and $d_{+}$ which play the roles of $x$ and $\partial$ in the discussion above, and proved that the unicellular LLT polynomial associated to a Dyck path $h$ can be computed as $d_h(1)$ where $d_h$ is the operator obtained from $h$ by replacing each east step with $d_{+}$ and each north step with $d_{-}$.\par

      In the recent paper \cite{AS2020} it is demonstrated how to obtain the chromatic symmetric function from a similar set of relations (see \cite[Corollary 6.16]{AS2020}). Actually, one of the relations used in loc. cit. is contained in the modular law used here. The authors also consider a \emph{bounce relation} on Schr\"oeder paths that implies the modular law and other relations.\par

\section{The algorithm}

      Our main goal in this section is to prove Theorem \ref{thm:main}. We first need some notation.
      We denote the set of Dyck paths by $\D$ and by $\D_n$ the set of Dyck paths of size $n$. We will also think of $\D$ as the set of Hessenberg functions via the identification between Dyck paths and Hessenberg functions. There is a (non-commutative) product on the set $\D$ given by concatenation of Dyck paths, while on Hessenberg functions the product of $h_1\col [n_1]\to[n_1]$ with $h_2\col [n_2]\to [n_2]$ is the function $h\col [n_1+n_2]\to [n_1+n_2]$ given by $h(i)=h_1(i)$ if $i\in [n_1]$ and $h(i)=h_2(i)+n_1$ if $i\in \{n_1+1,\ldots, n_1+n_2\}$. We denote this product by $h=h_1\cdot h_2$. We say that $h$ is irreducible if it cannot be written as a product of non-trivial Hessenberg functions, or equivalently, if the Dyck path associated to $h$ does not touch the diagonal. Every $h$ is written uniquely as the product of irreducible Hessenberg functions, which are called the irreducible components of $h$. There is an involution on $\D$ given by transposing the Dyck paths, and we denote by $h^t$ the transpose of $h$. We let $k_n$ be the unique Hessenberg function in $\D_n$ with $k_n(1)=n$, and call it \emph{complete}.\par
        Also, we define $[n]_q:=\frac{q^n-1}{q-1}$ and $n!_q=\prod_{j=1}^n [j]_q$  in $\mathbb{Q}(q)$ and let $\A$ be a $\mathbb{Q}(q)$-algebra.
      \begin{Def}
      \label{def:modular}
      We say that a function $f\col \D\to \A$ satisfies the  \emph{modular law} if
      \begin{equation}
      \label{eq:rec}
      (1+q)f(h_1)=qf(h_0)+f(h_2)
      \end{equation}
      whenever one of the following conditions hold
      \begin{enumerate}
          \item\label{item:1} There exists $i\in [n-1]$ such that $h_1(i-1)<h_1(i)<h_1(i+1)$ and $h_1(h_1(i))=h_1(h_1(i)+1)$ or $h_1(i)=n$. Moreover, $h_0$ and $h_2$ satisfy $h_k(j):=h_1(j)$ for every $j\neq i$ and $k=0,2$, while $h_k(i)=h_1(i)-1+k$.
          \item\label{item:2} There exists $i\in [n-1]$ such that $h_1(i+1)=h_1(i)+1$ and $h_1^{-1}(i)=\emptyset$. Moreover, $h_0$ and $h_2$ satisfy $h_k(j):=h_1(j)$ for every $j\neq i,i+1$ and $k=0,2$, while $h_0(i)=h_0(i+1)=h_1(i)$ and $h_2(i)=h_2(i+1)=h_1(i+1)$.
      \end{enumerate}
      \end{Def}
      We note that if we define the function $f^t\col \D\to \A$ by $f^t(h)=f(h^t)$, then $f$ satisfies the modular law if and only if $f^t$ satisfies it as well.  Throughout this section $f\col \D\to \A$ will be a function satisfying the modular law.\par
     Conditions \eqref{item:1} and \eqref{item:2} can also be seen in the associated Dyck paths (see \cite[Equation 12]{AS2020} for precise definitions).\par

    As in the introduction, we proceed by constructing more general relations starting from Equation \eqref{eq:rec}. This is the content of Propositions \ref{prop:basicrel},  \ref{prop:basicreldual} and \ref{prop:relations} below. The first two are proved in \cite[Theorem 3.4 (a)]{HNY}.

 \begin{Prop}
    \label{prop:basicrel}
    Let $h_1$ be a Hessenberg function and $1\leq i<j\leq n$ be integers such that
\begin{enumerate}
    \item either $h_1(i-1)<h_1(i)$, or $i=1$ and $h_1(1)>1$.
    \item $j-1<h_1(i)=h_1(i+1)=\ldots=h_1(j-1)$
    \item $h_1^{-1}(\{i,\ldots, j-2\})=\emptyset$.
\end{enumerate}
If
\[
h_0(l):=\begin{cases}
h_1(l)-1&\text{if }l\in \{i,\ldots, j-1\},\\
h_1(l)&\text{otherwise},
\end{cases}
\text{ and }
h_2(l):=\begin{cases}
h_1(l)-1 &\text{ if }l\in \{i,\ldots, j-2\},\\
h_1(l)&\text{ otherwise.}
\end{cases}
\]
Then
\[
f(h_1)=[j-i]_qf(h_2)+(1-[j-i]_q)f(h_0).
\]
\begin{figure}[htb]  \centering
\begin{tikzpicture}
\begin{scope}[scale=0.4]
\draw[help lines] (0,0) grid +(5,3);
\dyckpath{0,0}{1,0,0,0,0,0,1,1}
\node at (2.3,-1) {$h_1$};
\end{scope}
\begin{scope}[scale=0.4,shift={(7,0)}]
\draw[help lines] (0,0) grid +(5,3);
\dyckpath{0,0}{0,0,0,0,0,1,1,1}
\node at (2.3,-1) {$h_0$};
\end{scope}
\begin{scope}[scale=0.4, shift={(14,0)}]
\draw[help lines] (0,0) grid +(5,3);
\dyckpath{0,0}{0,0,0,0,1,0,1,1}
\node at (2.3,-1) {$h_2$};
\end{scope}
\end{tikzpicture}
\caption{The relevant pieces of the Dyck paths $h_1$, $h_0$ and $h_2$.}
\label{fig:h0h2a}
\end{figure}
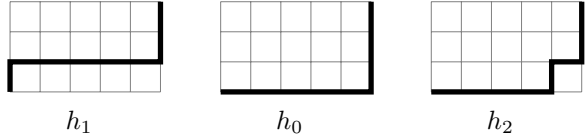
\end{Prop}
\begin{proof}
We define $g_k$, for $k\in \{0,\ldots, j-i-1\}$ as
\[
g_k(l):=\begin{cases}
h_1(l)-1 &\text{ if }l\in \{i,\ldots, i+k\}\\
h_1(l)& \text{ otherwise}.
\end{cases}
\]
Since $f$ satisfies the modular law, and each triple $g_{k}$, $g_{k+1}$ and $g_{k-1}$ satisfies condition \eqref{item:2}, we have that
\[
(1+q)f(g_k)=qf(g_{k+1})+f(g_{k-1}),
\]
and since $g_0=h_1$, $g_{j-i-2}=h_2$ and $g_{j-i-1}=h_0$, we get the result.
\end{proof}

\begin{Prop}
\label{prop:basicreldual}
Let $h_1$ be a Hessenberg function and $1\leq i\leq n$ be an integer such that
\begin{enumerate}
    \item either $h_1(i-1)<h_1(i)$, or $i=1$ and $h_1(1)>1$;
    \item there exist $1\leq a<h_1(i)$ such that $h_1(a+1)=h_1(a+2)=\ldots=h_1(h_1(i))$. (Usually, we will consider $a=h_1(i-1)$.)
\end{enumerate}
If
\[
h_0(l):=\begin{cases}
a &\text{ if }l=i\\
h_1(l)& \text{ otherwise}
\end{cases}
\;\text{ and }\;
h_2(l):=\begin{cases}
a+1 &\text{ if }l=i\\
h_1(l)& \text{ otherwise}
\end{cases}
\]
Then
\[
f(h_1)=[h(i)-a]_qf(h_2)+(1-[h(i)-a]_q)f(h_0).
\]
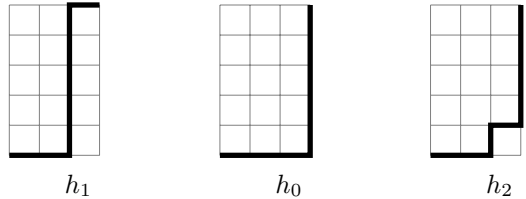
\begin{figure}[htb]  \centering
\begin{tikzpicture}
\begin{scope}[scale=0.4]
\draw[help lines] (0,0) grid +(3,5);
\dyckpath{0,0}{0,0,1,1,1,1,1,0}
\node at (2.3,-1) {$h_1$};
\end{scope}
\begin{scope}[scale=0.4,shift={(7,0)}]
\draw[help lines] (0,0) grid +(3,5);
\dyckpath{0,0}{0,0,0,1,1,1,1,1}
\node at (2.3,-1) {$h_0$};
\end{scope}
\begin{scope}[scale=0.4, shift={(14,0)}]
\draw[help lines] (0,0) grid +(3,5);
\dyckpath{0,0}{0,0,1,0,1,1,1,1}
\node at (2.3,-1) {$h_2$};
\end{scope}
\end{tikzpicture}
\caption{The relevant pieces of the Dyck paths $h_1$, $h_0$ and $h_2$.}
\label{fig:h0h2s}
\end{figure}
\end{Prop}
\begin{proof}
The proof is analogous to that of Proposition \ref{prop:basicrel}, using condition \eqref{item:1} in place of \eqref{item:2}.
\end{proof}

We now state the analogue of Equation \eqref{eq:relpol} for functions that satisfy the modular law.

\begin{Prop}
\label{prop:relations}
Let $h_1$ be a Hessenberg function and $1\leq i<j\leq n$ be integers such that
\begin{enumerate}
    \item\label{item:cond1} either $h_1(i-1)<h_1(i)$, or $i=1$ and $h_1(1)>1$.
    \item\label{item:cond2} $j-1<h_1(i)=h_1(i+1)=\ldots=h_1(j-1)<h_1(j)$.
    \item\label{item:cond3} $h_1^{-1}(\{i,\ldots, j-2\})=\emptyset$.
    \item\label{item:cond4} There exists $1\leq b\leq h_1(j)-h_1(i)$ such that $h_1(h_1(i))=h_1(h_1(i)+1)=h_1(h_1(i)+2)=\ldots=h_1(h_1(i)+b)$.
\end{enumerate}
If
\[
h_0(l):=\begin{cases}
h_1(l)-1 &\text{ if }l\in \{i,\ldots, j-1\}\\
h_1(l)& \text{ otherwise}
\end{cases}
\text{ and }
h_2(l):=\begin{cases}
h_1(l)-1 &\text{ if }l\in \{i,\ldots, j-2\}\\
h_1(l)+b & \text{ if }l=j-1\\
h_1(l)&\text{ otherwise}
\end{cases}
\]
Then
\begin{equation}
    \label{eq:proprelations}
[b+1]_qf(h_1)=[j-i]_qf(h_2)+([b+1]_q-[j-i]_q)f(h_0).
\end{equation}
\begin{figure}[htb]  \centering
\begin{tikzpicture}
\begin{scope}[scale=0.4]
\draw[help lines] (0,0) grid +(5,3);
\dyckpath{0,0}{1,0,0,0,0,0,1,1}
\node at (2.3,-1) {$h_1$};
\end{scope}
\begin{scope}[scale=0.4,shift={(7,0)}]
\draw[help lines] (0,0) grid +(5,3);
\dyckpath{0,0}{0,0,0,0,0,1,1,1}
\node at (2.3,-1) {$h_0$};
\end{scope}
\begin{scope}[scale=0.4, shift={(14,0)}]
\draw[help lines] (0,0) grid +(5,3);
\dyckpath{0,0}{0,0,0,0,1,1,1,0}
\node at (2.3,-1) {$h_2$};
\end{scope}
\end{tikzpicture}
\caption{The relevant pieces of the Dyck paths $h_1$, $h_0$ and $h_2$. Compare with Equation \eqref{eq:relpartial}.}
\label{fig:h0h2}
\end{figure}
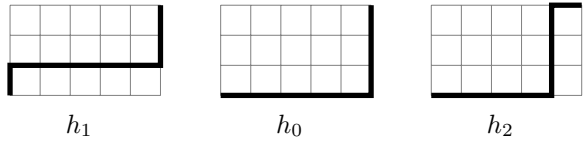
\end{Prop}

\begin{proof}
Define $h'$ (see Figure \ref{fig:h'}) as
\[
h'(l)=\begin{cases}
h_1(l)-1& \text{ if } l \in {i,\ldots,j-2}\\
h_1(l)& \text{ otherwise}
\end{cases}
\]

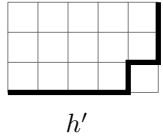
\begin{figure}[htb]  \centering
\begin{tikzpicture}
\begin{scope}[scale=0.4]
\draw[help lines] (0,0) grid +(5,3);
\dyckpath{0,0}{0,0,0,0,1,0,1,1}
\node at (2.3,-1) {$h'$};
\end{scope}
\end{tikzpicture}
\caption{The relevant piece of the Dyck path $h'$.}
\label{fig:h'}
\end{figure}

By Proposition \ref{prop:basicrel} applied to $h_1$ and Proposition \ref{prop:basicreldual}  applied to $h_2$, we have that
\begin{align*}
f(h_1)&=[j-i]_qf(h')+(1-[j-i]_q)f(h_0)\\
f(h_2)&=[k+1]_qf(h')+(1-[k+1]_q)f(h_0)
\end{align*}
Then
\[
[k+1]_qf(h_1)-[j-i]_qf(h_2)=([k+1]_q-[j-i]_q)f(h_0),
\]
and the result follows.
\end{proof}

 In terms of the Dyck path, the choice of $i,j,b$ determines a subpath $ne^{j-i}n^b$ of the Dyck path induced by $h$.  Condition \eqref{item:cond3} means that there are no east steps between the $i$-th and $(j-1)$-th north steps (see Figure \ref{fig:3}), while condition \eqref{item:cond4} means that there are no north steps between the $h(i)$-th and $(h(i)+b)$-th east steps (see Figure \ref{fig:4}). An example where both conditions are satisfied can be seen in Figure \ref{fig:sat34}.\par

\begin{figure}[htb]  \centering
\begin{minipage}{0.49\linewidth}
\begin{center}
\begin{tikzpicture}
\begin{scope}[scale=0.3]
\draw[help lines] (0,0) grid +(8,8);
\dyckpath{0,0}{1,1,0,1,1,0,0,1,1,0,0,1,1,0,0,0}
\dyckpathc{1,3}{1,0,0,1,1}{yellow}
\gline{0,0}{8,8}{blue}{1}
\gline{1,2}{2,2}{red}{0.5}
\gline{2,2}{2,4}{red}{0.5}
\node at (1.5,-0.5) {$i$};
\node at (3.5,-0.5) {$j$};
\node at (-0.7,5) {$b$};
\draw [decorate,decoration={brace,amplitude=3pt},xshift=-4pt,yshift=0pt] (0,4) -- (0,6);
\end{scope}
\end{tikzpicture}
\end{center}
\caption{An example of a choice of $i,j,b$ that does not satisfy condition \eqref{item:cond3}.}
\label{fig:3}
\end{minipage}
\begin{minipage}{0.49\linewidth}
\begin{center}
\begin{tikzpicture}
\begin{scope}[scale=0.3]
\draw[help lines] (0,0) grid +(8,8);
\dyckpath{0,0}{1,1,1,0,1,0,0,1,1,0,1,0,1,0,0,0}
\dyckpathc{1,3}{1,0,0,1}{yellow}
\gline{0,0}{8,8}{blue}{1}
\gline{3,4}{4,4}{red}{0.5}
\gline{4,4}{4,6}{red}{0.5}
\node at (1.5,-0.5) {$i$};
\node at (3.5,-0.5) {$j$};
\node at (-0.7,4.5) {$b$};
\draw [decorate,decoration={brace,amplitude=3pt},xshift=-4pt,yshift=0pt] (0,4) -- (0,5);
\end{scope}
\begin{scope}[scale=0.3, shift={(10,0)}]
\draw[help lines] (0,0) grid +(8,8);
\dyckpath{0,0}{1,1,1,0,1,0,1,1,1,0,0,0,1,0,0,0}
\dyckpathc{1,3}{1,0,1,1}{yellow}
\gline{0,0}{8,8}{blue}{1}
\gline{2,5}{5,5}{red}{0.5}
\gline{5,5}{5,7}{red}{0.5}
\node at (1.5,-0.5) {$i$};
\node at (2.5,-0.5) {$j$};
\node at (-0.7,5) {$b$};
\draw [decorate,decoration={brace,amplitude=3pt},xshift=-4pt,yshift=0pt] (0,4) -- (0,6);
\end{scope}
\end{tikzpicture}
\end{center}
\caption{Examples of two choices of $i,j,b$ that does not satisfy condition \eqref{item:cond4}.}
\label{fig:4}
\end{minipage}
\end{figure}

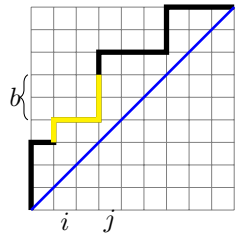
\begin{figure}[htb]  \centering
\begin{tikzpicture}
\begin{scope}[scale=0.3]
\draw[help lines] (0,0) grid +(9,9);
\dyckpath{0,0}{1,1,1,0,1,0,0,1,1,1,0,0,0,1,1,0,0,0}
\dyckpathc{1,3}{1,0,0,1,1}{yellow}
\gline{0,0}{9,9}{blue}{1}
\node at (1.5,-0.5) {$i$};
\node at (3.5,-0.5) {$j$};
\node at (-0.7,5) {$b$};
\draw [decorate,decoration={brace,amplitude=3pt},xshift=-4pt,yshift=0pt] (0,4) -- (0,6);
\end{scope}
\end{tikzpicture}
\caption{An example of a choice of $i,j,b$ that satisfy conditions \eqref{item:cond3} and \eqref{item:cond4}.}
\label{fig:sat34}
\end{figure}

The idea is that, for every $h\in \D_n$ with at least one irreducible component that is not complete, there always exist integers $i,j,b$ satisfying the conditions in Proposition \ref{prop:relations}. So we can reduce the computation of $f(h_1)$ to the computations of $f(h_0)$ and $f(h_2)$. As long as we always choose the greatest possible $b$, this process actually terminates.\par
  To see that there always exist such $i,j,b$, we say that a Hessenberg function $h$ is \emph{aligned} if, for every $i=1,\ldots, n$, we have that either $h(h(i)+1)>h(h(i))$ or $h(h(i))=n$ (see Figure \ref{fig:ali}). We note that a Hessenberg function is aligned if and only if its irreducible components are aligned as well.
\begin{figure}[htb]  \centering
\begin{tikzpicture}
\begin{scope}[scale=0.8]
\begin{scope}[scale=0.4]
\draw[help lines] (0,0) grid +(7,7);
\dyckpath{0,0}{1,1,0,1,1,1,0,1,0,0,0,1,0,0}
\gline{0,0}{7,7}{blue}{1}
\gline{1,2}{2,2}{red}{0.5}
\gline{2,5}{2,2}{red}{0.5}
\gline{2,5}{5,5}{red}{0.5}
\gline{5,6}{5,5}{red}{0.5}
\node at (3.5,-1) {Aligned};
\end{scope}
\begin{scope}[scale=0.4,shift={(10,0)}]
\draw[help lines] (0,0) grid +(7,7);
\dyckpath{0,0}{1,1,1,0,1,0,1,0,1,0,1,0,0,0}
\gline{0,0}{7,7}{blue}{1}
\gline{1,3}{3,3}{red}{0.5}
\gline{3,5}{3,3}{red}{0.5}
\gline{2,4}{4,4}{red}{0.5}
\gline{4,6}{4,4}{red}{0.5}
\node at (3.5,-1) {Aligned};
\end{scope}
\begin{scope}[scale=0.4, shift={(20,0)}]
\draw[help lines] (0,0) grid +(7,7);
\dyckpath{0,0}{1,1,0,1,1,1,0,0,1,0,0,1,0,0}
\gline{0,0}{7,7}{blue}{1}
\gline{1,2}{2,2}{red}{0.5}
\gline{3,5}{3,3}{red}{0.5}
\node at (3.5,-1) {Non aligned};
\end{scope}
\end{scope}
\end{tikzpicture}
\caption{An aligned and non aligned Hessenberg function}
\label{fig:ali}
\end{figure}
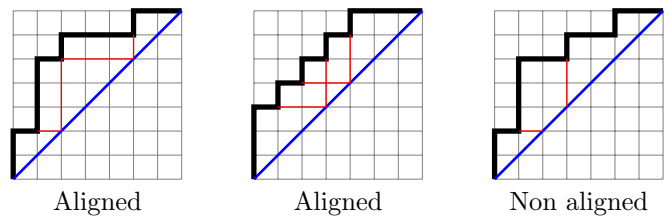

\begin{Prop}
\label{prop:ijk1}
Let $h\in \D_n$ be a non aligned Hessenberg function. Let $i$ be the smallest integer such that $h(h(i)+1)=h(h(i))$ and $h(h(i))<n$. Define $j:=\min\{l; l>i, h(l)>h(i)\}$ and $b:=\max\{l;l\geq 1, h(h(i)+l)=h(i), h(i)+l\leq h(j)\}$. Then $h$ and $i,j,b$ satisfy the conditions in Proposition \ref{prop:relations}.
\end{Prop}
\begin{proof}
The only condition that is not straightforward to check is condition \eqref{item:cond3} in Proposition \ref{prop:relations}. If there exists $l$ such that $h(l)\in \{i,\ldots, j-2\}$, then $h(l)+1\in \{i,\ldots, j-1\}$, which means that $h(h(l)+1)=h(h(l))$ which condradicts the minimality of $i$. See figure \ref{fig:condnalig}.
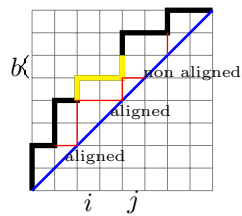
\begin{figure}[htb]  \centering
\begin{tikzpicture}
\begin{scope}[scale=0.3]
\draw[help lines] (0,0) grid +(8,8);
\dyckpath{0,0}{1,1,0,1,1,0,1,0,0,1,1,0,0,1,0,0}
\dyckpathc{2,4}{1,0,0,1}{yellow}
\gline{0,0}{8,8}{blue}{1}
\gline{1,2}{2,2}{red}{0.5}
\gline{2,4}{2,2}{red}{0.5}
\gline{2,4}{4,4}{red}{0.5}
\gline{4,5}{4,4}{red}{0.5}
\gline{4,5}{5,5}{red}{0.5}
\gline{6,7}{6,6}{red}{0.5}
\node at (2.8,1.5) {\tiny aligned};
\node at (4.8,3.5) {\tiny aligned};
\node at (7.1,5.2) {\tiny non aligned};
\node at (2.5,-0.5) {$i$};
\node at (4.5,-0.5) {$j$};
\node at (-0.7,5.5) {$b$};
\draw [decorate,decoration={brace,amplitude=3pt},xshift=-4pt,yshift=0pt] (0,5) -- (0,6);
\end{scope}
\end{tikzpicture}
\caption{The smallest $i$ such that $h(h(i)+1)=h(h(i))$.}
\label{fig:condnalig}
\end{figure}

\end{proof}
\begin{Prop}
\label{prop:ijk2}
Let $h\in \D_n$ be an aligned irreducible non-complete Hessenberg function. Let $1\leq i< j\leq n$ be the integers such that $h(i-1)<h(i)=h(i+1)=\ldots h(j-1)<h(j)=n$ (here define $h(0)=0$). Define $b:=n-h(i)$. Then $h$ and $i,j,b$ satisfy the conditions in Proposition \ref{prop:relations}.
\end{Prop}
\begin{proof}
 The only condition that is not straightforward to check is condition \eqref{item:cond3} in Proposition \ref{prop:relations}. If there exists $l$ such that $h(l)\in \{i,\ldots, j-2\}$, then $h(l)+1\in \{i,\ldots, j-1\}$, which means that $h(h(l)+1)=h(h(l))<n$ which contradicts the fact that $h$ is aligned. See figure \ref{fig:condalig}.
\begin{figure}[htb]  \centering
\begin{tikzpicture}
\begin{scope}[scale=0.3]
\draw[help lines] (0,0) grid +(8,8);
\dyckpath{0,0}{1,1,0,1,1,1,0,1,0,0,0,1,1,0,0,0}
\dyckpathc{2,5}{1,0,0,0,1,1}{yellow}
\gline{0,0}{8,8}{blue}{1}
\gline{1,2}{2,2}{red}{0.5}
\gline{2,5}{2,2}{red}{0.5}
\gline{2,5}{5,5}{red}{0.5}
\gline{5,6}{5,5}{red}{0.5}
\node at (2.5,-0.5) {$i$};
\node at (5.5,-0.5) {$j$};
\node at (-0.7,7) {$b$};
\draw [decorate,decoration={brace,amplitude=3pt},xshift=-4pt,yshift=0pt] (0,6) -- (0,8);
\end{scope}
\end{tikzpicture}
\caption{Depiction of $i,j,b$ when $h$ is aligned.}
\label{fig:condalig}
\end{figure}
\end{proof}

\begin{Alg}
\label{alg:alg}
We can now define the following algorithm for computing $f$ in terms of $f(k_{n_1}\cdot\ldots\cdot k_{n_m})$.
\begin{enumerate}
    \item\label{item:step1} If $h=k_{n_1}\cdot\ldots\cdot k_{n_m}$ for some positive integers $n_1,\ldots, n_m$, then return $f(h)$. Else, go to step \eqref{item:step2}.
    \item\label{item:step2} If $h$ is non-aligned, choose $i,j,b$ as in Proposition \ref{prop:ijk1} and use Proposition \ref{prop:relations}. Return to step \eqref{item:step1} with $h_0$ and $h_2$. Else, go to step \eqref{item:step3}.
    \item\label{item:step3} If $h$ is aligned, choose an irreducible component of $h$ which is not complete. Then choose $i,j,b$ as in Proposition \ref{prop:ijk2} for this component and use Proposition \ref{prop:relations}. Return to step \eqref{item:step1} with $h_0$ and $h_2$.
\end{enumerate}
If $f$ is multiplicative, that is if $f(h_1\cdot h_2)=f(h_1)f(h_2)$, then we can compute $f$ in terms of $f(k_n)$.
\end{Alg}
\begin{Exa}
In Figure \ref{fig:alg} we show the steps of Algorithm \ref{alg:alg} for the Hessenberg function $h=(2,4,4,5,5)$.
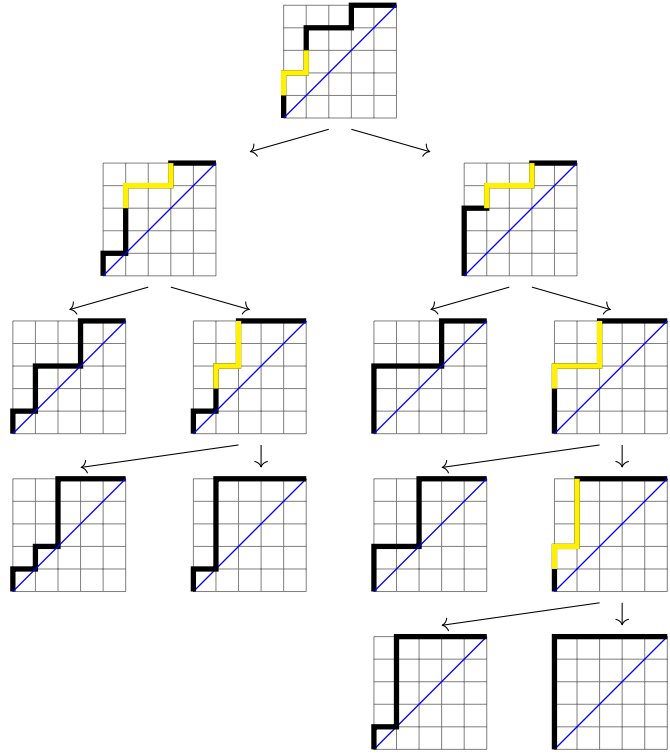
\begin{figure}[htb]  \centering
    \begin{tikzpicture}
    \begin{scope}[scale=0.5]
       \begin{scope}[scale=0.6,shift={(-8,0)}]
       \draw[help lines] (0,0) grid +(5,5);
       \dyckpath{0,0}{1,1,0,1,1,0,0,1,0,0}
       \dyckpathc{0,1}{1,0,1}{yellow}
       \gline{0,0}{5,5}{blue}{0.5}
       \end{scope}
       \begin{scope}[scale=0.6,shift={(-6,-0.5)}]
       \draw[->] (0,0) -- (-3.5,-1);
       \draw[->] (1,0) -- (4.5,-1);
       \end{scope}
       \begin{scope}[scale=0.6,shift={(-8,-7)}]
       \begin{scope}[shift={(-8,0)}]
       \draw[help lines] (0,0) grid +(5,5);
       \dyckpath{0,0}{1,0,1,1,1,0,0,1,0,0}
       \gline{0,0}{5,5}{blue}{0.5}
       \dyckpathc{1,3}{1,0,0,1}{yellow}
       \end{scope}
       \begin{scope}[shift={(8,0)}]
       \draw[help lines] (0,0) grid +(5,5);
       \dyckpath{0,0}{1,1,1,0,1,0,0,1,0,0}
       \gline{0,0}{5,5}{blue}{0.5}
       \dyckpathc{1,3}{1,0,0,1}{yellow}
       \end{scope}
       \end{scope}
       \begin{scope}[scale=0.6,shift={(2,-7.5)}]
       \draw[->] (0,0) -- (-3.5,-1);
       \draw[->] (1,0) -- (4.5,-1);
       \end{scope}
       \begin{scope}[scale=0.6,shift={(-20,-14)}]
       \begin{scope}
       \draw[help lines] (0,0) grid +(5,5);
       \dyckpath{0,0}{1,0,1,1,0,0,1,1,0,0}
       \gline{0,0}{5,5}{blue}{0.5}
       \end{scope}
       \begin{scope}[shift={(8,0)}]
       \draw[help lines] (0,0) grid +(5,5);
       \dyckpath{0,0}{1,0,1,1,0,1,1,0,0,0}
       \gline{0,0}{5,5}{blue}{0.5}
       \dyckpathc{1,2}{1,0,1,1}{yellow}
       \end{scope}
       \end{scope}
       \begin{scope}[scale=0.6,shift={(-14,-7.5)}]
       \draw[->] (0,0) -- (-3.5,-1);
       \draw[->] (1,0) -- (4.5,-1);
       \end{scope}
       \begin{scope}[scale=0.6,shift={(-4,-14)}]
       \begin{scope}
       \draw[help lines] (0,0) grid +(5,5);
       \dyckpath{0,0}{1,1,1,0,0,0,1,1,0,0}
       \gline{0,0}{5,5}{blue}{0.5}
       \end{scope}
       \begin{scope}[shift={(8,0)}]
       \draw[help lines] (0,0) grid +(5,5);
       \dyckpath{0,0}{1,1,1,0,0,1,1,0,0,0}
       \gline{0,0}{5,5}{blue}{0.5}
       \dyckpathc{0,2}{1,0,0,1,1}{yellow}
       \end{scope}
       \end{scope}
       \begin{scope}[scale=0.6,shift={(6,-14.5)}]
       \draw[->] (0,0) -- (-7,-1);
       \draw[->] (1,0) -- (1,-1);
       \end{scope}
       \begin{scope}[scale=0.6,shift={(-4,-21)}]
       \begin{scope}
       \draw[help lines] (0,0) grid +(5,5);
       \dyckpath{0,0}{1,1,0,0,1,1,1,0,0,0}
       \gline{0,0}{5,5}{blue}{0.5}
       \end{scope}
       \begin{scope}[shift={(8,0)}]
       \draw[help lines] (0,0) grid +(5,5);
       \dyckpath{0,0}{1,1,0,1,1,1,0,0,0,0}
       \gline{0,0}{5,5}{blue}{0.5}
       \dyckpathc{0,1}{1,0,1,1,1}{yellow}
       \end{scope}
       \end{scope}
       \begin{scope}[scale=0.6,shift={(6,-21.5)}]
       \draw[->] (0,0) -- (-7,-1);
       \draw[->] (1,0) -- (1,-1);
       \end{scope}
       \begin{scope}[scale=0.6,shift={(-20,-21)}]
       \begin{scope}
       \draw[help lines] (0,0) grid +(5,5);
       \dyckpath{0,0}{1,0,1,0,1,1,1,0,0,0}
       \gline{0,0}{5,5}{blue}{0.5}
       \end{scope}
       \begin{scope}[shift={(8,0)}]
       \draw[help lines] (0,0) grid +(5,5);
       \gline{0,0}{5,5}{blue}{0.5}
       \dyckpath{0,0}{1,0,1,1,1,1,0,0,0,0}
       \end{scope}
       \end{scope}
       \begin{scope}[scale=0.6,shift={(-10,-14.5)}]
       \draw[->] (0,0) -- (-7,-1);
       \draw[->] (1,0) -- (1,-1);
       \end{scope}
       \begin{scope}[scale=0.6,shift={(-4,-28)}]
       \begin{scope}
       \draw[help lines] (0,0) grid +(5,5);
       \gline{0,0}{5,5}{blue}{0.5}
       \dyckpath{0,0}{1,0,1,1,1,1,0,0,0,0}
       \end{scope}
       \begin{scope}[shift={(8,0)}]
       \draw[help lines] (0,0) grid +(5,5);
       \gline{0,0}{5,5}{blue}{0.5}
       \dyckpath{0,0}{1,1,1,1,1,0,0,0,0,0}
       \end{scope}
       \end{scope}
       \end{scope}
    \end{tikzpicture}
    \caption{Running Algorithm \ref{alg:alg} for $h=(2,4,4,5,5)$.}
    \label{fig:alg}
    \end{figure}
\end{Exa}

\begin{Rem}
  If we define $\widetilde{f}(h)=\frac{f(h)}{q^{\ell(h)/2}}$, where $\ell(h)=\sum_{i=1}^n h(i)-i$, then Equation \eqref{eq:rec} has the following, more symmetric, form
  \[
  (q^{-\frac{1}{2}}+q^{\frac{1}{2}})\widetilde{f}(h_1)=q^{-\frac{1}{2}}\widetilde{f}(h_0)+q^{\frac{1}{2}}\widetilde{f}(h_2).
  \]
  Likewise, one can derive symmetric forms for the equations in Propositions \ref{prop:basicrel}, \ref{prop:basicreldual} and \ref{prop:relations}. For instance, Equation \eqref{eq:proprelations} in Proposition \eqref{prop:relations} will read
  \[
  \llbracket b+1\rrbracket_q\widetilde{f}(h_1)=\llbracket j-i\rrbracket_q\widetilde{f}(h_2)+\llbracket b+1-j+i\rrbracket_q\widetilde{f}(h_0)
  \]
  where
  \[
  \llbracket a\rrbracket_q=\frac{q^{\frac{a}{2}}-q^{-\frac{a}{2}}}{q^{\frac{1}{2}}-q^{-\frac{1}{2}}}.
  \]
  When $\A=A(q)$ for some $\mathbb{Q}$-algebra $A$, we say that $g\in A(q)$ is \emph{palindromic with center of symmetry $k$} if
  \[
  q^{2k}g(q^{-1})=g(q).
  \]
  By the discussion above,  when $f(k_{n_1}\cdot\ldots \cdot k_{n_m})$ is palindromic with center of symmetry 
  \[
  \frac{1}{2}\sum_{j=1}^m \frac{n_j(n_j-1)}{2}
  \]
  for every sequence of positive integers $(n_1,\ldots, n_m)$, Algorithm \ref{alg:alg} proves that  $f(h)$ is palindromic with center of symmetry $\frac{\ell(h)}{2}$ for every $h\in \D$. For example, this holds true for the chromatic quasisymmetric function.\par
 \end{Rem} 
  
\begin{Rem}
When $h$ is abelian, that is, if $h(h(1)+1)=n$, one can compute $f(h)$ from the modular law using the algorithm in \cite[Proposition 29]{Alexandersson_2020}.\par
\end{Rem}
\begin{Rem}
\label{rem:network}
If $h$ is abelian then $h$ is aligned and all subsequent Hessenberg functions appearing in Algorithm \ref{alg:alg} will be aligned as well. This means that no steps of type \eqref{item:step2} will be required. In this case the algorithm can be represented graphically by a planar network as described below. \par
  First, we need a definition. We let $a_{i,j}:=\min\{k\geq 0; h(i-k)<j\}$ if $h(1)<j$ and $a_{i,j}=i$ otherwise. This network can be visualized in the plane lattice. It has starting point $(i,h(i))$, where $i:=\max\{j,h(j)<n\}$ (or $(0,n)$ when $h(1)=n$). It has steps $(0,-1)$ and $(-1,-1)$ and end points $(j,j)$ with $j\geq0$.\par
   Each step $(i,j)\to (i-1,j-1)$ has weight $\frac{[a_{i,j}]_q}{[n-j+1]_q}$, while each step $(i,j)\to (i,j-1)$ has weight $1-\frac{[a_{i,j}]_q}{[n-j+1]_q}$. Finally, each end point $(j,j)$ has weight $f(k_j\cdot k_{n-j})$. The network is depicted in Figure \ref{fig:network}.
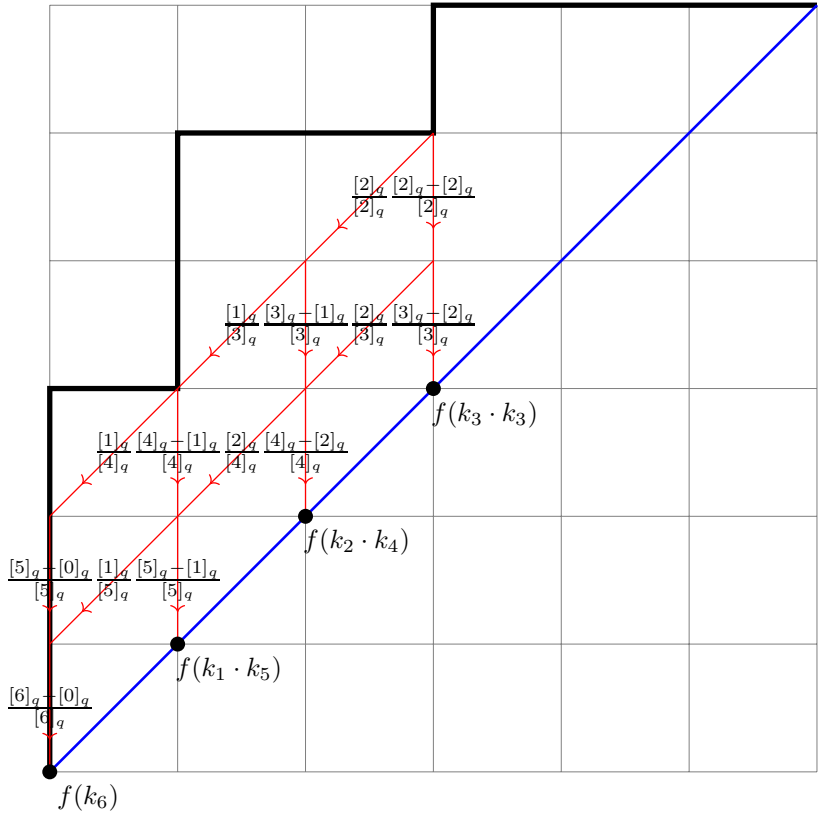
\begin{figure}[htb]  \centering
\begin{tikzpicture}
\begin{scope}[scale=1.7, decoration={
    markings,
    mark=at position 0.75 with {\arrow{>}}}
    ]
\draw[help lines] (0,0) grid +(6,6);
\dyckpath{0,0}{1,1,1,0,1,1,0,0,1,0,0,0}
\glinear{3,5}{red}{0.5}{$\frac{[2]_q-[2]_q}{[2]_q}$}
\glinearr{3,5}{red}{0.5}{$\frac{[2]_q}{[2]_q}$}
\glinear{3,4}{red}{0.5}{$\frac{[3]_q-[2]_q}{[3]_q}$}
\glinearr{3,4}{red}{0.5}{$\frac{[2]_q}{[3]_q}$}
\glinear{2,4}{red}{0.5}{$\frac{[3]_q-[1]_q}{[3]_q}$}
\glinearr{2,4}{red}{0.5}{$\frac{[1]_q}{[3]_q}$}
\glinear{2,3}{red}{0.5}{$\frac{[4]_q-[2]_q}{[4]_q}$}
\glinearr{2,3}{red}{0.5}{$\frac{[2]_q}{[4]_q}$}
\glinear{1,3}{red}{0.5}{$\frac{[4]_q-[1]_q}{[4]_q}$}
\glinearr{1,3}{red}{0.5}{$\frac{[1]_q}{[4]_q}$}
\glinear{1,2}{red}{0.5}{$\frac{[5]_q-[1]_q}{[5]_q}$}
\glinearr{1,2}{red}{0.5}{$\frac{[1]_q}{[5]_q}$}
\glinear{0,1}{red}{0.5}{$\frac{[6]_q-[0]_q}{[6]_q}$}
\glinear{0,2}{red}{0.5}{$\frac{[5]_q-[0]_q}{[5]_q}$}
\gline{0,0}{6,6}{blue}{1}
\node at  (3,3) [shape=circle, fill=black, inner sep=2pt] {};
\node at (3.4,2.8) {$f(k_3\cdot k_3)$};
\node at  (2,2) [shape=circle, fill=black, inner sep=2pt] {};
\node at (2.4,1.8) {$f(k_2\cdot k_4)$};
\node at  (1,1) [shape=circle, fill=black, inner sep=2pt] {};
\node at (1.4,0.8) {$f(k_1\cdot k_5)$};
\node at  (0,0) [shape=circle, fill=black, inner sep=2pt] {};
\node at (0.3,-0.2) {$f(k_6)$};
\end{scope}
\end{tikzpicture}
\caption{The planar network for the Hessenberg function $h=(3,5,5,6,6,6)$.}
\label{fig:network}
\end{figure}
\end{Rem}

\begin{Rem}
\label{rem:networkpos}
There is one case where this network is manifestly ``positive'' (in the sense that the numerators of the weights are $q$-polynomials with non-negative coefficients). This happens when the largest clique contained in the indifference graph associated to $h$ contains the vertex $n$. More precisely, let $j_0:=\min\{j;h(j)=n\}$ and assume that $n-j_0\geq h(i)-i$ for every $i \in [n]$. Then for $i<j_0$ the following inequality holds
\[
h(i-n+j-1)-i+n-j+1\leq n-j_0
\]
and hence
\[
h(i-n+j-1)\leq j+i-1-j_0<j.
\]
This implies that $a_{i,j}\leq n-j+1$ and then all weights appearing in the network for $h$ are positive.
\end{Rem}

\begin{Thm}
\label{thm:algterm}
Algorithm \ref{alg:alg} terminates.
\end{Thm}
\begin{proof}
We first introduce some notation. For a Hessenberg function $h_1\in \D_n$ we say that the quadruple $(i,j,b,g)$ is a \emph{step} for $h_1$, if $h_1$ and $i,j,b$ satisfy the conditions in Proposition \ref{prop:relations} (with maximum possible $b$) and $g$ is either $h_0$ or $h_2$. Moreover, if $g=h_k$, for $k=0,2$, we say that $(i,j,b,g)$ is a \emph{$k$-step} for $h_1$. The height of a step $(i,j,b,g)$ for $h_1$ is defined as $h_1(i)+b$.

Since the set $\D_{n}$ is finite it is enough to prove that if we start with a Hessenberg function $g$, it never reappears in the steps of the algorithm. Suppose, for contradiction, that there exists $N>0$, positive integers $i_l,j_l, b_l$ for $l=1,\ldots, N-1$, and Hessenberg functions $g_l$ for $l=1,\ldots N$, such that $g_1=g$, $g_N=g$ and $(i_l,j_l,b_l,g_{l+1})$ is a step for $g_l$ for every $l$. Also, let $M$ be the maximum height attained, that is, $M=\max\{ g_l(j_l)+b_l;l\in [N-1]\}$, and define $a_l:=\min\{i; g_l(i)\geq M\}$. See Figure \ref{fig:Ma} below.
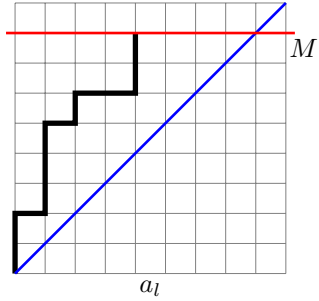
\begin{figure}[htb]  \centering
\begin{tikzpicture}
\begin{scope}[scale=0.4]
\draw[help lines] (0,0) grid +(9,9);
\dyckpath{0,0}{1,1,0,1,1,1,0,1,0,0,1,1}
\gline{0,0}{9,9}{blue}{1}
\gline{-0.3,8}{9.3,8}{red}{1}
\node at (9.6,7.5) {$M$};
\node at (4.5,-0.5) {$a_l$};
\end{scope}
\end{tikzpicture}
\caption{Depiction of the maximum height $M$ and of $a_l$.}
\label{fig:Ma}
\end{figure}

By Proposition \ref{prop:relations} and the fact that $M$ is the maximum height we see that $a_l$ is non increasing. Since $a_1=a_n$, we have that $(a_l)$ is constant and we set $a:=a_1$. The idea now is to prove that $g_l(a-1)$ is non-increasing, which would imply that $g_N(a-1)<g_1(a-1)$, a contradiction.\par



Let $c_l:=\max\{ j< M; g_l(j+1)>g_l(j)\}$. First of all, we note that $c_l$ actually exists, because $g_l(a-1)<M\leq g_l(a)$, in particular $c_l\geq a-1$. Moreover, $c_l$ is independent of $l$, because we have that $g_{l+1}(i)=g_l(i)$ whenever $i\geq a$.


  Since $g_l(c)<g_l(c+1)$, we cannot have that  $g_l(i_l)\leq c<g_l(i_l)+b_l$ (see Condition (4) in Proposition \ref{prop:relations}). This means that either $g_l(i_l)> c$ or $g_l(i_l)+b\leq c$.   In either case, we have that the sets $B_l:=\{i, h_l(i)\leq c\}$ and $C_l:=\{i, c< h_l(i)\}$ satisfy that  $B_{l}\subset B_{l+1}$ and $C_{l+1}\subset C_l$. Then these sequences of sets are also constant sequences and we set $B:=B_1$ and $C:=C_1$.\par

  Since there is at least one step $(i_l,j_l,b_l,g_{l+1})$ of height $M$ and such a step must satisfy $j_l=a$, we must have that $g_l(a-1)>c$, In particular $a-1\in C$, which implies that $h_l(a-1)>c$ for every $l$. Therefore, every step with $j_l=a$ is of height $M$ (recall that in Propositions \ref{prop:ijk1} and \ref{prop:ijk2} we choose $b$ maximal). However, we know that $g_l(a-1)<M$ for every $l$ and that $g_l(a-1)\neq g_{l+1}(a-1)$ only when $j_l=a$. On the other hand every step of height $M$ must be a $0$-step, which means that $g_{l+1}(a-1)<g_l(a-1)$, therefore $g_n(a-1)<g_1(a-1)$ which is a contradiction.
\end{proof}

Finally, to complete the proof of Theorem \ref{thm:main} we need the following proposition.
\begin{Prop}
For every sequence of positive integers $(c_1,\ldots, c_m)$ and permutation $\sigma\in S_m$, we have that
\[
f(k_{c_1}\cdot k_{c_2}\cdot\ldots \cdot k_{c_m})=f(k_{c_{\sigma(1)}}\cdot k_{c_{\sigma(2)}}\cdot\ldots\cdot k_{c_{\sigma(m)}}).
\]
In particular, if $h$ is a Hessenberg function and $h^t$ is its transpose, then $f(h)=f(h^t)$.
\end{Prop}
\begin{proof}
The second statement follows from the first. Defining $f^t$ as $f^t(h)=f(h^t)$, then $f^t$ also satisfies the modular law. Since
\[
f^t(k_{c_1}\cdot k_{c_2}\cdot\ldots \cdot k_{c_m})=f(k_{c_m}\cdot k_{c_{m-1}}\cdot\ldots \cdot k_{c_1})=f(k_{c_1}\cdot k_{c_2}\cdot\ldots \cdot k_{c_m})
\]
for every sequence of positive integers $(c_1,\ldots, c_m)$, we have that $f^t=f$ by Algorithm \ref{alg:alg} and Theorem \ref{thm:algterm}, which proves that $f(h^t)=f(h)$.\par

To prove the first statement, we begin by showing that $f(k_a\cdot k_{n-a})=f(k_{n-a}\cdot k_a)$ for every $a\in [n]$. Clearly, we can assume that $a\leq n/2$. We fix $n$ and proceed by induction on $a$.  If $a=0$, there is nothing to prove. Otherwise, consider the Hessenberg function $h$ such that $h(1)=h(2)=\ldots=h(a)=n-a$ and $h(a+1)=n$. Since $h$ is abelian we can apply Remark \ref{rem:network} to write $f(h)=\sum_{i\leq a} c_i f(k_i\cdot k_{n-i})$. However, applying the same Remark to $f^t$ and noticing that $h=h^t$, we get that $f(h)=\sum_{i\leq a} c_if(k_{n-i}\cdot k_i)$. Since, by induction hypothesis, we have that $f(k_i\cdot k_{n-i})=f(k_{n-i}\cdot k_i)$ for $i<a$, and since
\[
c_a=\prod_{i=1}^{n-2a} \frac{[a+i]_q-[a]_q}{[a+i]_q}\neq 0,
\]
we get that $f(k_a\cdot k_{n-a})=f(k_{n-a}\cdot k_a)$. \par
  We can actually generalize the argument above, and prove that $f(h_1\cdot k_a\cdot k_{n-a}\cdot h_2)=f(h_1\cdot k_{n-a}\cdot k_a\cdot h_2)$ for every $h_1,h_2\in \D$. Since every permutation is a product of simple transpositions we have the stated result.
\end{proof}

We finish this section with some remarks. First, if we work with the $q$-Weyl algebra and substitute $\partial$ with the $q$-derivation
\[
\partial_q f:=\frac{f(x)-f(xq)}{x(1-q)}
\]
a straightforward computation gives the $q$-analogues of Equations \eqref{eq:basicpartial} and \eqref{eq:relpartial} (although we have to substitute $q$ with $q^{-1}$).
\begin{align*}
    (1+q^{-1})x\partial_qx=&q^{-1}\partial_qx^2+x^2\partial_q\\
    [b+1]_{q^{-1}}x\partial_q^lx^b=&[l]_{q^{-1}}\partial_q^{l-1}x^{b+1}\partial_q+([b+1]_{q^{-1}}-[l]_{q^{-1}})\partial_q^lx^{b+1}.
\end{align*}

Second, Proposition \ref{prop:relations} has a more general form, which originates from the following $q$-Chu-Vandermonde equality
\[
{{a+b}\choose a}_{q^{-1}}x^a\partial_q^lx^b=\sum_{j=0}^a q^{-j(b-l+j)}{{a+b-l} \choose {a-j}}_{q^{-1}}{l\choose j}_{q^{-1}}\partial_q^{l-j}x^{a+b}\partial_q^{j}.
\]
Since we did not need this general form, we merely state it and leave it as an exercise to the avid reader.
\begin{Prop}
Let $h\in \D_n$ be a Hessenberg function and $1\leq i<j\leq n$ and $a\geq 1$ be integers such that
\begin{enumerate}
    \item either $h(i-1)+a\leq h(i)$, or $i=1$ and $h(1)>a$.
    \item $j-1<h(i)=h(i+1)=\ldots=h(j-1)<h(j)$.
    \item $h^{-1}(\{i,\ldots, j-2\})=\emptyset$.
    \item There exists $1\leq b\leq h(j)-h(i)$ such that
    \[
    h(h(i)-a+1)=h(h(i)-a+2)=\ldots=h(h(i))=h(h(i)+1)=\ldots=h(h(i)+b).
    \]
\end{enumerate}
If
\[
h_k(l):=\begin{cases}
h(l)-a &\text{ if }l\in \{i,\ldots, j-1-k\}\\
h(l)+b & \text{ if }l\in\{j-1-k,\ldots, j-1\}\\
h(l)&\text{ otherwise}
\end{cases}
\]
then
\begin{equation}
\label{eq:cv}
{a+b\choose a}_qf(h)=\sum_{j=0}^a q^{j(b-l+j)}{a+b-l\choose a- j}_q{l\choose j}_q f(h_j).
\end{equation}
\end{Prop}

\section{The chromatic quasisymmetric function}

    We begin by recalling that the chromatic quasisymmetric function does indeed satisfy the modular law.  This is already well known in the literature (see for instance \cite[Proposition 3.1]{GPmodular}, \cite[Theorem 3.4]{Lee}, \cite[Theorem 3.1]{HNY} and \cite[Corollary 20 and Proposition 23]{Alexandersson_2020}). In particular, if we write
    \[
    \csf_q(h)=\sum_{\lambda\vdash n} \csf_{q,\lambda}(h)e_{\lambda}
    \]
    we have that the functions $\csf_{q,\lambda}\col D\to \mathbb{Q}(q)$ also satisfy the modular law. The following result proves that every other function $f\col \D\to \mathbb{Q}(q)$ is actually a $\mathbb{Q}(q)$-linear combination of $\csf_{q,\lambda}$.\par
       Let $V_n$ be the space of functions $f\col \D_n\to \mathbb{Q}(q)$ that satisfy the modular law.

    \begin{Thm}
    The space $V_n$ has basis $\{\csf_{q,\lambda}\}_{\lambda\vdash n}$.
    \end{Thm}
    \begin{proof}
    By Theorem \ref{thm:main}, we have that $V_n$ has dimension at most the number of partitions of $n$. On the other hand, we have that $\csf_{q,\lambda}(k_{\mu})=0$ for every $\mu\vdash n$ with $\mu\neq \lambda$ and $\csf_{q,\lambda}(k_{\lambda})=\prod_{i=1}^{\ell(\lambda)} \lambda_i!_q$, which means the $\csf_{q,\lambda}$ are linear independent. This finishes the proof.
    \end{proof}

    In a way, the coefficients of the chromatic quasisymmetric function in the elementary basis are the simplest functions that satisfy the modular law, as the following makes precise.\par
    \begin{Cor}
    \label{cor:fcsf}
    Let $\A$ be a $\mathbb{Q}(t)$-algebra, and $f\col \D\to \A$ be a function satisfying the modular law. Then
    \[
    f(h)=\sum_{\lambda\vdash n}\frac{\csf_{q,\lambda}(h)}{\lambda!_q}f(k_\lambda).
    \]
    \end{Cor}
   We will now prove that, when $h$ is abelian, Algorithm \ref{alg:alg} is manifestly positive after we change $h$ to a suitable abelian Hessenberg function that has the same chromatic quasisymmetric function. As a consequence, we recover some of the results in \cite{HaradaPrecup} and \cite{Cho}. \par

   For a Hessenberg function $h\in \D_n$ we call the sequence $a_i=h(i)-i$ the \emph{area sequence} of $h$. We note that in order for a sequence of non-negative integers to be an area sequence of a Hessenberg function we must have $a_i+i\leq n$ and $a_{i+1}\geq a_{i}-1$. Given a Hessenberg function $h$, note that the chromatic polynomial of its indifference graph $G$ is given by $\chi_G(x)=\prod_{i=1}^n(x-a_i)$ and as such it does not depend on the order of the elements in the area sequence.  We have a similar result for the chromatic quasisymmetric function.

\begin{Prop}
\label{prop:abelcsf}
Let $h_1$ and $h_2$ be abelian Hessenberg functions in $\D_n$ such that the area sequence of $h_1$ is a permutation of the area sequence of $h_2$. Then $\csf_q(h_1)=\csf_q(h_2)$.
\end{Prop}
\begin{proof}
The main idea is that for abelian Hessenberg functions the chromatic symmetric function is uniquely determined by the chromatic polynomial, which in turn is determined by the elements in the area sequence.\par
  Let $\chi_q\col \D\to \mathbb{Q}(t)[x]$ be given by $\chi_q(h)=\prod (x-[h(i)-i]_q)$. Then, $\chi_q$ satisfies the modular law. Indeed, it is enough to notice that
  \[
  (1+q)(x-[a]_q)=q(x-[a-1]_q)+(x-[a+1]_q)
  \]
  for every $a\geq 1$.  \par
  Consider the homomorphism $\alpha\col \Lambda_q\to \mathbb{Q}(q)[x]$ between $\mathbb{Q}(t)$-algebras given by
  \[
  \alpha(e_n)=\frac{\prod_{j=0}^{n-1} (x-[j]_q)}{n!_q}.
  \]
  We can see that
  \[
  \alpha(e_{n-i,i})=\frac{\prod_{j=0}^{n-i-1} (x-[j]_q)}{(n-i)!_q}\frac{\prod_{j=0}^{i-1} (x-[j]_q)}{i!_q},
  \]
  which means that $\{\alpha(e_{n-i,i})\}_{i=0,\ldots, \lceil n/2\rceil}$ is a linearly independent set over $\mathbb{Q}(t)$ (one can just compute the values at $x=[n-1]_q,[n-2]_q,\ldots, [\lceil n/2\rceil-1]_q$).\par
    Since $\alpha\circ \csf_q$ and $\chi_q$ are multiplicative, satisfy the modular law, and agree at the complete graphs, they must coincide, $\alpha\circ \csf_q=\chi_q$. If $h_1$ and $h_2$ are abelian Hessenberg functions with the same area sequence, then $\chi_q(h_1)=\chi_q(h_2)$, which implies that $\alpha(\csf_q(h_1)-\csf_q(h_2))=0$. However, $\csf_q(h_1)-\csf_q(h_2)$ is a $\mathbb{Q}(t)$-linear combination of $e_{n-i,i}$ (see \cite[Theorem 6.3]{ShareshianWachs}). By the linear independence of  $\alpha(e_{n-i,i})$, we have that $\csf_q(h_1)=\csf_q(h_2)$.
\end{proof}
The following corollary is evident from Corollary \ref{cor:fcsf}.
\begin{Cor}
Let $h_1$ and $h_2$ be abelian Hessenberg functions in $\D_n$ such that the area sequence of $h_1$ is a permutation of the area sequence of $h_2$. Then $f(h_1)=f(h_2)$ for every function $f\col \D\to \A$ satisfying the modular law.
\end{Cor}

We remark that $h_1=(2,4,4,5,5)$ and $h_2=(3,3,4,5,5)$ have area sequences $(1,2,1,1,0)$ and $(2,1,1,1,0)$, but different chromatic quasisymmetric functions, which is possible because $h_1$ is not abelian. \par
 We now prove that after rearranging the area sequence, we get an abelian Hessenberg function for which all steps in Algorithm \ref{alg:alg} are positive.

\begin{Lem}
\label{lem:abelian}
Given an abelian Hessenberg function $h_1\in \D_n$, there exists $h_2\in \D_n$, also abelian, with area sequence $(b_i)$ having the following two properties:
\begin{enumerate}
    \item The sequence $(b_i)$ is a permutation of the area sequence of $h_1$.
    \item We have that $b_1\geq b_i$ for every $i\in [n]$.
\end{enumerate}
\end{Lem}
\begin{proof}
We begin noting that the area sequence of $h^t$ is a permutation of the area sequence of $h$ for every $h\in \D$. This is clear from the fact that the associated indifference graphs are isomorphic and hence have tha same chromatic polynomial.\par

  Let $a_i$ be the area sequence of $h_1$ and define $j_0:=\min\{j\in [n], h_1(j)=n\}$. Up to taking transposes, we may assume that $a_1\geq a_{j_0}$. Define the sequence $(b_i)$ such that $b_1,\ldots, b_{j_0}-1$ is a non-increasing permutation of $a_1,\ldots, a_{j_0-1}$ and $b_i=a_i$ for $i=j_0,\ldots, n$. Clearly $b_1\geq b_i$ for every $i\in [j_0-1]$ and $b_1\geq a_1\geq a_{j_0}=b_{i_0}> b_i$ for every $i=j_0+1,\ldots, n$. All that is left to prove is that $(b_i)$ is the area sequence of an abelian Hessenberg function. \par
  First, we prove that $(b_i)$ induces a Hessenberg function $h_2(i)=b_i+i$. This means proving that $b_i+i\leq n$ and $b_{i+1}\geq b_i-1$, both of which we prove by contradiction.  If $b_i+i>n$ for some $i\in [n]$, then $i\in [j_0-1]$, since $b_i=a_i$ for $i=j_0,\ldots, n$. Since there must exist $l\in\{i,\ldots, j_0-1\}$ and $j\leq i$ such that $a_l=b_j$, we have that $a_l+l=b_j+l\geq b_i+i>n$, which is a contradiction.\par
   Now, assume that there exists $i$ such that $b_i\geq 2+b_{i+1}$. This means that there exists $j\in [j_0-1]$ such that $a_{j}+1\neq a_{i}$ for every $i\in [j_0-1]$  and $a_j<\max\{a_i;i\in [j_0-1]\}$. Since $a_{i+1}\geq a_i-1$ for every $i\in [n]$, this can only happen if $a_j=\max\{a_i.i\in [j]\}$, in particular $a_j\geq a_1\geq a_{j_0}$. However, every value between $\max\{a_i;i\in [j_o-1]\}$ and $a_{j_0}$ must appear in $(a_i)_{i\in [j_0]}$, which proves that there exists $i$ such that $a_j+1=a_i$.
\end{proof}

\begin{Prop}
\label{prop:csfunimodal}
If $h$ is an abelian Hessenberg function then $\csf_q(h)$ is $e$-positive and $e$-unimodal. That is, the coefficients of $\csf_q$ in the elementary basis are unimodal polynomials in $q$ with positive coefficients.
\end{Prop}
\begin{proof}
By Lemma \ref{lem:abelian} and Proposition \ref{prop:abelcsf}  above, we have that there exists an abelian Hessenberg function $h'$ with $\csf_q(h')=\csf_q(h)$ and such that its area sequence $(b_i)$ satisfies $b_1\geq b_i$.  This implies that $h'^t$ satisfies the condition in Remark \ref{rem:networkpos}. Hence, substituting $h$ with $h'^t$, we can assume that $h$ satisfies the condition in Remark \ref{rem:networkpos}.\par

As in Remark \ref{rem:network}, we let $i=\max\{j; h(j)<n\}$. Each path from $(i,h(i))$ to a final point $(l,l)$ has final denominator $(n-l)!_q/(n-h[(i))!_q$. Since $\csf_q(k_l\cdot k_{n-l})=l!_q(n-l)!_q e_{n-l,l}$, we have that the coefficient of $e_{n-l,l}$ is $\csf_q(h)$ is $l!_q(n-h(i))!_q\cdot P_l$, where $P_l$ is obtained from the numerators in the network. In this cases, Remark \ref{rem:networkpos} proves that $\csf_q(h)$ is $e$-positive.\par
To prove unimodality, it is sufficient to show that every path from $(i,h(i))$ to $(l,l)$ has unimodal contribution with the same center. Indeed, we recall that the product of two palindromic unimodal polynomials with centers $a$ and $b$ is a palindromic unimodal polynomial with center $a+b$ ,and the sum of two palindromic unimodal polynomials with same center $a$ still is a palindromic unimodal polynomial with center $a$ (see \cite{Stanley86}). By the discussion in previous paragraph, we can consider only the numerators in the network described in Remark \ref{rem:network}.\par

 We note that the difference $[a]_q-[b]_q$, for $a>b$, is a palindromic unimodal polynomial with center $\frac{a+b-1}{2}$,  when $b=0$, we get that $[a]_q$ is a palindromic unimodal polynomial with center $\frac{a-1}{2}$. In particular, every path from $(i,h(i))$ to $(l,l)$ has unimodal contribution. Moreover, every such path can be obtained from any other path (from $(i,h(i))$ to $(l,l)$) by successively replacing a subpath  $(-1,-1),(0,-1)$ with $(0,-1),(-1,-1)$ (or vice-versa), that is, replacing the red path with the blue path (or vice-versa) in Figure \ref{fig:subpath}.
\begin{figure}[htb]
    \centering
    \begin{tikzpicture}[scale=2, decoration={
    markings,
    mark=at position 0.5 with {\arrow{>}}}]
    \node at  (1,2) [shape=circle, fill=black, inner sep=2pt] {};
    \node at  (1,1) [shape=circle, fill=black, inner sep=2pt] {};
    \node at  (0,1) [shape=circle, fill=black, inner sep=2pt] {};
    \node at  (0,0) [shape=circle, fill=black, inner sep=2pt] {};
    \node at (0.4,1.7) {$[a]_q$};
    \node at (2,1.5) {$([n-j+1]_q-[a]_q)$};
    \node at (-1.2,0.5) {$([n-(j-1)+1]_q-[b-1]_q)$};
    \node at (0.6,0.3) {$[b]_q$};
    \draw[color=red,  postaction=decorate] (0,1) -- (0,0);
    \draw[color=red,  postaction=decorate] (1,2) -- (0,1);
    \draw[color=blue, postaction=decorate] (1,1) -- (0,0);
    \draw[color=blue,  postaction=decorate] (1,2) -- (1,1);
    \end{tikzpicture}
    \caption{Caption}
    \label{fig:subpath}
\end{figure}

Assuming that $(i,j)$ is the starting point in Figure \ref{fig:subpath}, and that $a_{i,j}=a$ and $a_{i,j-1}=b$, we have that the (numerators of the) weights in each edge are the ones depicted in Figure \ref{fig:subpath}. However, both products $[a]_q([n-(j-1)+1]_q-[b-1]_q)$  and $([n-j+1]_q-[a]_q)[b]_q$ are palindromic unimodal polynomials of the same center. Indeed, we have that
\[
\frac{a-1}{2}+\frac{n-(j-1)+1+(b-1)-1}{2}=\frac{n-j+a+b-1}{2}=\frac{n-j+1+a-1}{2}+\frac{b-1}{2}.
\]
This proves that every path from $(i,h(i))$ to $(l,l)$ gives a palindromic unimodal polynomial with the same center, hence the sum $P_l$ of all these contributions will remain palindromic and unimodal. This means that the coefficient $l!_q(n-h(i))!_qP_l$ is a palindromic unimodal polynomial.
\end{proof}

We finish this section with a remark concerning the Hopf algebra of Dyck paths.

\begin{Rem}
  Abusing notation, we denote by $\mathcal{D}$ the Hopf algebra of Dyck paths defined in \cite{GP}. As a vector space over $\mathbb{Q}(t)$, the algebra $\D$ has the set of Dyck paths as a basis, with multiplication induced by concatenation. The comultiplication is a little more involved and we refer the reader to \cite{GP}.\par
  One can run Algorithm \ref{alg:alg} in $\mathcal{D}$. That is, if $I\subset \mathcal{D}$ is the vector space generated by relations $(1+q)h_1-qh_0-h_2$ whenever $h_0,h_1,h_2$ are Dyck paths satisfying one of the conditions in Definition \ref{def:modular}, then
\[
h\equiv \sum_{\lambda\vdash n}\frac{\csf_{q,\lambda}(h)}{\lambda!_q}k_{\lambda}\mod I
\]
for every Dyck path $h$, where $\lambda!_q=\prod_{i=1}^{\ell(\lambda)}\lambda_i!q$. In particular, we have $I=\ker(\csf_q)$ when $\csf_q$ is regarded as a Hopf algebra map $\csf_q\col \mathcal{D}\to \Lambda_q$.
\end{Rem}

\section{Rook placements and $q$-hit numbers}
  In this section, we prove a $q$-analogue of \cite[Theorem 4.3]{StanStem}. We mention that rook placements also appear in relation with terms in the $e$-expansion of $\csf_q$, see \cite{AlexPanova}.  First we must recall the definition of the $q$-analogue of hit numbers introduced in \cite{GarsiaRemmel}.\par
Let $E_m$ be the $m\times m$ board and let $\lambda$ be a partition such that its Young diagram fits in $E_m$, that is $\lambda_1,\ell(\lambda)\leq m$. To keep the notation consistent with the last sections, we will number the lines of $E_m$ from bottom to top. This means that the Young diagram of $\lambda$ is the set of cells $\{(i,j); j\leq \lambda_{m+1-i}\}$.\par
Define $B_{j,m}(\lambda)$ as the set of placements of $m$ rooks on $E_m$ such that precisely $j$ are in the Young diagram of $\lambda$. Each rook placement has a \emph{$\lambda$-weight} defined as in \cite{Dworkin}. This weight is the number of cells $e\in E_m$ such that
\begin{enumerate}
    \item there is no rook on $e$,
    \item there is no rook to the left of $e$,
    \item one of the following holds
    \begin{enumerate}
        \item if $e\in \lambda$ then the rook on the same column of $e$ is in $\lambda$ and below $e$,
        \item if $e\notin \lambda$ then the rook on the same column of $e$ is either in $\lambda$ or below $e$.
    \end{enumerate}
\end{enumerate} See Figure \ref{fig:rook} for an example, where the black circles are the rooks, while the white circles correspond to the cells $e$ satisfying the conditions above. 
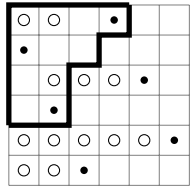
\begin{figure}[htb]  \centering
    \begin{tikzpicture}
\begin{scope}[scale=0.4]
\draw[help lines] (0,0) grid +(6,6);
\dyckpath{0,2}{0,0,1,1,0,1,0,1}
\dyckpath{0,2}{1,1,1,1,0,0,0,0}
\node at  (0.5,4.5) [shape=circle, fill=black, inner sep=1pt] {};
\node at  (1.5,2.5) [shape=circle, fill=black, inner sep=1pt] {};
\node at  (2.5,0.5) [shape=circle, fill=black, inner sep=1pt] {};
\node at  (3.5,5.5) [shape=circle, fill=black, inner sep=1pt] {};
\node at  (4.5,3.5) [shape=circle, fill=black, inner sep=1pt] {};
\node at  (5.5,1.5) [shape=circle, fill=black, inner sep=1pt] {};
\node at  (0.5,5.5) [shape=circle, draw, inner sep=1.5pt] {};
\node at  (0.5,1.5) [shape=circle, draw, inner sep=1.5pt] {};
\node at  (0.5,0.5) [shape=circle, draw, inner sep=1.5pt] {};
\node at  (1.5,5.5) [shape=circle, draw, inner sep=1.5pt] {};
\node at  (1.5,3.5) [shape=circle, draw, inner sep=1.5pt] {};
\node at  (1.5,1.5) [shape=circle, draw, inner sep=1.5pt] {};
\node at  (1.5,0.5) [shape=circle, draw, inner sep=1.5pt] {};
\node at  (2.5,1.5) [shape=circle, draw, inner sep=1.5pt] {};
\node at  (2.5,3.5) [shape=circle, draw, inner sep=1.5pt] {};
\node at  (3.5,3.5) [shape=circle, draw, inner sep=1.5pt] {};
\node at  (3.5,1.5) [shape=circle, draw, inner sep=1.5pt] {};
\node at  (4.5,1.5) [shape=circle, draw, inner sep=1.5pt] {};
\end{scope}
\end{tikzpicture}
    \caption{A rook placement in $B_{3,6}(\lambda)$ with $\lambda$-weight $12$ for $\lambda=(4,3,2,2)$.}
    \label{fig:rook}
\end{figure}

We then define
\[
R_{j,m}(\lambda):=\sum_{\sigma\in B_{j,m}(\lambda) } q^{\wt_{\lambda}(\sigma)}.
\]
where $\wt_{\lambda}(\sigma)$ is the $\lambda$-weight of $\sigma$.
\begin{Lem}
\label{lem:Rjm}
Let $\lambda$ be a partition in $E_m$. Suppose that either $j=\ell(\lambda)$ and $\lambda_1\leq m-1$ or $j=\lambda_1$ and $\ell(\lambda)\leq m-1$. Then
\[
R_{j,m}(\lambda)=([m]_q-[j]_q)R_{j,m-1}(\lambda).
\]
\end{Lem}
\begin{proof}
Assume first that $j=\ell(\lambda)$ and $\lambda_1\leq m-1$. Then the placements in $B_{j,m}(\lambda)$ are such that the rooks in the first $j$-lines are in $\lambda$, so we only have to place $m-j$ rooks on the remaining $(m-j)\times (m-j)$ board.

If a rook is placed in $\lambda$ on column $k$, we can actually compute the contribution to the $\lambda$-weight of the cells not in $\lambda$ on column $k$. This is precisely $m-j-a_k$, where $a_k$ is the number columns to the left of column $k$ with no rooks on $\lambda$. In particular it does not depend on how to place the remaining $m-j$ rooks. The same holds true for $B_{j,m-1}$. Hence, we have that
\[
R_{j,m}(\lambda)=t^j[m-j]_qR_{j,m-1}(\lambda).
\]
The factor $t^j$ comes from the differences $(m-j-a_k)-(m-1-j-a_k)=1$ for each rook in $\lambda$, and $[m-j]_q$ comes from the ratio $(m-j)!_q/(m-j-1)!_q$.\par
 When $j=\lambda_1$, we just note that $R_{j,m}(\lambda)=R_{j,m}(\lambda^t)$ (this follows from the deletion-contraction recurrence \cite[Corollary 6.12]{Dworkin}), and the result follows.
\end{proof}
\begin{Lem}
\label{lem:Rjrel}
Let $\lambda$ be a partition in $E_m$ such that there exists $i$ such that $\lambda_{i+1}<\lambda_{i}<\lambda_{i-1}$ (where we assume $\lambda_0=\infty$). Consider the partitions $\mu$ and $\sigma$ obtained from $\lambda$ by removing the cell $(m+1-i,\lambda_i)$ and adding the cell $(m+1-i,\lambda_i+1)$, respectively. Then
\begin{align*}
(1+q)R_{j,m}(\lambda)=&qR_{j,m}(\mu)+R_{j,m}(\sigma),\\
(1+q)R_{j,m}(\lambda^{t})=&qR_{j,m}(\mu^{t})+R_{j,m}(\sigma^{t}).
\end{align*}

\end{Lem}
\begin{proof}
We let $\ol{\lambda}$ be the contraction of the cell $(m+1-i,\lambda_i)$ of $\lambda$  (this means we remove the line $m+1-i$ and column $\lambda_i$ from $\lambda$). This is the same partition obtained from the contraction of the cell $(m+1-i,\lambda_i+1)$ of $\sigma$. By \cite[Theorem 6.11]{Dworkin}, we have that
\begin{align*}
R_{j,m}(\lambda)=qR_{j,m}(\mu)+R_{j-1,m-1}(\ol{\lambda})-q^nR_{j,m-1}(\ol{\lambda})\\
R_{j,m}(\sigma)=qR_{j,m}(\lambda)+R_{j-1,m-1}(\ol{\lambda})-q^nR_{j,m-1}(\ol{\lambda}),\\
\end{align*}
from which the first equality follows. The same argument holds for the transposes.
\end{proof}

Let $h\col[n]\to[n]$ be a Hessenberg function. Define the associated partition $\lambda$ as the partition such that $\lambda^{t}_i=n-h(i)$. When $h$ is abelian, this means that $\lambda$ is \emph{small} in the sense of \cite{StanStem}, i.e., $\lambda_1+\ell(\lambda)\leq n$. We define
\[
b_{j}(h):=\begin{cases}
q^j[n-2j]_qR_{j,n-j-1}(\lambda)& \text{ if }j\leq \lambda(1),\ell(\lambda)\leq n-j-1,\\
R_{j,n-j}(\lambda)&\text{ if }\{\lambda(1),\ell(\lambda)\}=\{j,n-j\},\\
0& \text{ otherwise }
\end{cases}
\]

\begin{Thm}
\label{thm:qhit}
If $h$ is an abelian Hessenberg function then
\[
\csf_q(h)=\sum_{j\leq n/2}j!_qb_{j}(h)e_{n-j,j}.
\]
\end{Thm}
\begin{proof}
By Lemmas \ref{lem:Rjm} and \ref{lem:Rjrel}, the right-hand side satisfies the modular law. Hence, by Theorem \ref{thm:main} and Remark \ref{rem:network}, it is enough to prove the equality when $h=k_m\cdot k_{n-m}$. We assume without loss of generality that $m\leq n/2$.  In this case the partition associated to $h$ is $\lambda=(m,m,\ldots,m)$. Moreover, we have that $\csf_q(h)=m!_q(n-m)!_qe_{n-m,m}$, while $b_{j}(h)=0$ if $j\neq m$ and $b_{m}(h)=R_{m,n-m}(\lambda)=(m-n)!_q$. The result follows.
\end{proof}
 By \cite[Theorem 6]{Hag98} we have that $q$-hit numbers are unimodal, this gives a different proof of Proposition \ref{prop:csfunimodal}. Vice-versa, Proposition \ref{prop:csfunimodal} and Theorem \ref{thm:qhit} give a different proof of the unimodality of $q$-hit numbers.

\section{Logarithmic concavity}
In this section we discuss the logarithmic concavity of the coefficients of $\csf_q$.  We need a few definitions first. We say that a polynomial $P(q)=\sum a_jq^j$ is \emph{log-concave with no internal zeros} if it is unimodal and $a_j^2\geq a_{j-1}a_{j+1}$ for every $j$. For simplicity, in what follows we will write \emph{log-concave} instead of \emph{log-concave with no internal zeros}. It is a well know fact that the product of two log-concave polynomials still is a log-concave polynomial (see \cite{Stanley86}). On the other hand, the sum of two log-concave polynomials does not need to be log-concave. \par
  We say that two log-concave polynomials $P_1(q)=\sum a_jq^j$ and $P_2(q)=\sum b_jq^j$ are \emph{synchronized} if 
  \begin{equation}
      \label{eq:partialsync}
      a_kb_k\geq a_{k+1}b_{k-1}\text{ and } a_kb_k\geq a_{k-1}b_{k+1}
  \end{equation}
 for every $ k\geq1$.  By \cite{GMTW} we have that if $P_1$, $P_2$ and $F$ are log-concave polynomials and $P_1$ and $P_2$ are synchronized, then $P_1\cdot F$ and $P_2\cdot F$ are synchronized as well. Moreover, if $P_1,\ldots, P_n$ are log-concave polynomials which are pairwise synchronized then $P_1+P_2+\ldots+P_n$ is a log-concave polynomial, see \cite[Theorem 2.13]{GMTW}.\par
  We note that the polynomials $[m]_q$ are log-concave with no internal zeros for every non-negative integer $m$.
\begin{Lem}
\label{lem:nmsync}
  We have that $[n]_q[m]_q$ and $[n+1]_q[m-1]_q$ are synchronized log-concave polynomials for every non-negative integers $n,m$. 
\end{Lem}
\begin{proof}
Both polynomials are products of log-concave polynomials, and hence are log-concave.\par
  Without loss of generality we will assume that $n\geq m$, in this case we have
  \begin{align*}
      [n]_q[m]_q=&1+2q+\ldots+(m-1)q^{m-2} +mq^{m-1}+mq^{m}+\ldots mq^{n-1}\\
                 &+(m-1)q^{n}+(m-2)q^{n+1}+\ldots+ q^{n+m-2},\\
      [n+1]_q[m-1]_q=&1+2q+\ldots +(m-1)q^{m-2}+(m-1)q^{m-1}+(m-1)q^m+\ldots (m-1)q^{n-1}\\
                     &+(m-1)q^{n}+(m-2)q^{n+1}+\ldots+ q^{n+m-2}.
  \end{align*}
  From these expressions, it is easy to check that the coefficients of both polynomials do satisfy Equation \eqref{eq:partialsync}. Indeed,  Equation \eqref{eq:partialsync} for $k=m-2, m-1, \ldots, n $ would follow from the inequalities bellow
  \begin{align*}
      (m-1)^2\geq& m(m-2),& (m-1)^2\geq& (m-2)(m-1),\\
      m(m-1)\geq & m(m-1), &  m(m-1)\geq& (m-1)^2.
  \end{align*}
    For $k<m-2$ or $k>n$, Equation \eqref{eq:partialsync} become $k^2\geq (k-1)(k+1)$, which is also clear.      
\end{proof}
  We can then prove the log-concavity of the $e$-coefficients of $\csf_q(G)$ when $G$ is the path graph, that is, the graph associated to the Hessenberg function $(2,3,\ldots, n-1,n,n)$.
  
  \begin{Prop}
  \label{prop:path}
  If $h=(2,3,\ldots, n-1,n,n)$, then the $e$-coefficients of $\csf_q(h)$ are log-concave polynomials.
  \end{Prop}
   \begin{proof}
   We have that the coefficient of $e_{\lambda}$ in $\csf_q(h)$ is (see \cite{Haiman})
   \begin{equation}
       \label{eq:pathformula}
       q^{\ell(\lambda)-1}\sum_{i=1}^{\ell(\lambda)} [\lambda_i]_q\prod_{j=1, j\neq i}^{\ell(\lambda)} [\lambda_j-1]_q.
   \end{equation}
   
  By Lemma \ref{lem:nmsync}, we have that $[\lambda_{i_1}]_q[\lambda_{i_2}-1]_q$ and $[\lambda_{i_1}-1]_q[\lambda_{i_2}]_q$ are synchronized for every $i_1,i_2\in \{1,\ldots, \ell(\lambda)\}$. Moreover, we have that $\prod_{j\neq i_1,i_2}[\lambda_j-1]$ is log-concave (because each of its factors is), this implies that $[\lambda_{i_1}]_q\prod_{j\neq i_1} [\lambda_j-1]_q$ and $[\lambda_{i_2}]_q\prod_{j\neq i_2} [\lambda_j-1]_q$ are  synchronized. By \cite[Theorem 2.13]{GMTW} we have that the sum in Equation \eqref{eq:pathformula} is log-concave.
  \end{proof}
  
  There are a few other special cases where closed formulas for the $e$-coefficients are available. In \cite[Theorem 4.2]{Cho} a formula for $\csf_q(h)$ is given when $h=(r,\ldots, r,n,n,\ldots,n)$. In this case, the $e$-coefficients of $\csf_q(h)$ are products of polynomials of the form $[m]_q$, in particular these coefficients are log-concave. \par 
  In \cite[Proposition 4.4]{HNY} a closed formula for $\csf_q(h)$ is given when $h$ corresponds to a lollipop graph, that is, $h=(2,3,4,\ldots, n+1,n+m,n+m,\ldots, n+m)$. If $m> n$ and $\lambda$ is a partition with $n+m>\lambda_1\geq m$, then the coefficient of  $\csf_q(h)$ in $e_{\lambda}$ is
    \[
       q^{\ell(\lambda)-1}(m-1)!_q[\lambda_1-1]_q\sum_{i=2}^{\ell(\lambda)} [\lambda_i]_q\prod_{j=2, j\neq i}^{\ell(\lambda)} [\lambda_j-1]_q.
   \]
 This is essentially the same sum appearing in Equation \eqref{eq:pathformula}, which is log-concave. If $\lambda_1=n+m$, then the coefficient of $e_{\lambda}$ is $[n+m]_q(m-1)!_q$, which is log-concave. All other coefficients are $0$.\par
   A more careful analysis of the proof of Proposition \ref{prop:csfunimodal} could lead to a proof of logarithmic concavity for the abelian case. 
  
  We have also tested every indifference graph up to 12 vertices, and all have log-concave $e$-coefficients.

\begin{Conjecture}
The coefficients of $\csf_q(h)$ in the elementary basis are log-concave polynomials for every $h$.
\end{Conjecture}

In the Schur basis we have that the coefficient of $s_{6,1,1,1}$ in the expansion of $\csf_q(h)$, for $h=(3,4,4,4,5,6,7,8,9)$, is given by
\[
q^5 + 3q^4 + 10q^3 + 10q^2 + 3q + 1,
\]
which is not log-concave. For an example involving an irreducible Hessenberg function, we have that the coefficient of $s_{4,2,2,1,1,1}$ in the expansion of $\csf_q(h)$, for $h=(2,3,4,5,6,9,10,11,11,11,11)$,  is given by 
\[
q^{14} + 5q^{13} + 28q^{12} + 100q^{11} + 227q^{10} + 349q^9 + 349q^8 + 227q^7 + 100q^6 + 28q^5 + 5q^4 + q^3,
\]
which is not log-concave. Curiously, up to taking its transposed, this is the only irreducible Hessenberg function $h\col[n]\to [n]$, with $n\leq 11$, that has a Schur coefficient that is not log-concave.  \par
 In the power sum basis, we have that the coefficient of $p_{1,1,1,1,1}$ in the expansion of $\csf_q(h)$, for $h=(3,4,5,5,5)$, is given by
 \[
\frac{1}{120}(q^7+4q^6+17q^5+38q^4+38q^3+17q^2+4q+1),
 \]
which is not log-concave.

\bibliographystyle{amsalpha}

\bibliography{bibhess}

\end{document}